\documentclass[12pt]{article}

\usepackage[utf8]{inputenc}
\usepackage[english]{babel}
\usepackage{amsthm}
\usepackage{amsmath}
\usepackage{amsfonts}
\usepackage{tikz}
\usepackage{appendix}
\usepackage{enumerate}
\usepackage[normalem]{ulem}
\usepackage{mathtools}
\usepackage{stmaryrd}
\usepackage[noadjust]{cite}
\usepackage{cases}
\usepackage{caption}
\DeclareCaptionLabelSeparator{none}{ }
\usepackage[left=2.5cm, right=2.5cm, top=2.5cm, bottom=2.5cm]{geometry}

\usepackage[shortlabels]{enumitem}

\usepackage[hidelinks, bookmarks, bookmarksnumbered, pdfstartview={XYZ null null 1.00}]{hyperref}

\newtheorem{theorem}{Theorem}[section]
\newtheorem{lemma}[theorem]{Lemma}
\newtheorem{proposition}[theorem]{Proposition}
\newtheorem{corollary}[theorem]{Corollary}
\theoremstyle{definition}
\newtheorem{definition}[theorem]{Definition}
\newtheorem{remark}[theorem]{Remark}

\newcommand{\suchthat}{\ifnum\currentgrouptype=16 \mathrel{}\middle|\mathrel{}\else\mid\fi}

\DeclareMathOperator*{\esssup}{ess\,sup}

\DeclarePairedDelimiter{\abs}{\lvert}{\rvert}
\DeclarePairedDelimiter{\norm}{\lVert}{\rVert}

\DeclareMathOperator{\Ran}{Ran}
\DeclareMathOperator{\Span}{Span}
\DeclareMathOperator{\diag}{diag}
\DeclareMathOperator{\rank}{rank}

\newcommand{\vertii}[1]{{\left\vert\kern-0.25ex\left\vert #1 
        \right\vert\kern-0.25ex\right\vert}}

\theoremstyle{definition}

\begin{document}

\title{Approximate and exact controllability criteria for 
linear one-dimensional hyperbolic
systems }

\author{Yacine Chitour\thanks{Universit\'e Paris-Saclay, CNRS, CentraleSup\'elec, Laboratoire des signaux et syst\`emes, 91190, Gif-sur-Yvette, France.} \and Sébastien Fueyo\thanks{School of Electrical Engineering, 
   Tel Aviv University, Ramat Aviv 69978, Israel.} \and Guilherme Mazanti\thanks{Universit\'e Paris-Saclay, CNRS, CentraleSup\'elec, Inria, Laboratoire des signaux et syst\`emes, 91190, Gif-sur-Yvette, France.} \and Mario Sigalotti\thanks{Sorbonne Universit\'e, Inria, CNRS, Laboratoire Jacques-Louis Lions (LJLL), F-75005 Paris, France.}}
\maketitle

\begin{abstract}
This paper deals with the controllability of linear one-dimensional hyperbolic systems.
Reformulating the problem in terms of linear difference equations and making use of infinite-dimensional realization theory, we obtain both necessary and 
sufficient conditions 
for approximate and exact controllability, 
expressed in the frequency domain. The results are applied to flows in 
networks.

\vspace{2em}

\noindent Keywords: linear one-dimensional hyperbolic systems, boundary controls, linear difference equations, approximate controllability, exact controllability, flows in networks, realization theory.

\end{abstract}

\section{Introduction}
\label{sec:intro}
Linear one-dimensional  hyperbolic systems are frequently used to model many systems such as traffic flows or electronic circuits (see, e.g., the monograph \cite{Suarez}). The properties of such hyperbolic systems, such as their stability,  stabilizability, and controllability, have been studied intensively in the literature \cite{bastin2016stability, CoNg, CoronNguyenoptimal, GaravelloHanPiccoli, Auriol2022Robust}. In this paper, we focus on their approximate and exact controllability. More precisely, we consider the 
linear one-dimensional hyperbolic system
\begin{subequations}
\label{eq:hyperbolic}
\begin{numcases}{} 
\label{eq_hyp}
\partial_t R(t,x)+\Lambda(x) \partial_x R(t,x)+D(x)R(t,x)=0,& $t > 0,\ \  x \in (0, 1)$, \\
\label{eq_bord}
\begin{pmatrix}
R^+(t,0) \\
R^{-}(t,1)
\end{pmatrix}=M \begin{pmatrix}
R^+(t,1) \\
R^{-}(t,0)
\end{pmatrix}+Bu(t),& $t \ge 0$,
\end{numcases}
\end{subequations}
where the $n\times n$ matrices $\Lambda(x)$ and $D(x)$ are diagonal  and, for $\Lambda(x)$, with nonzero diagonal entries whose sign is independent of $x$; $R^+$ (respectively, $R^-$) gathers the components of $R$ whose corresponding diagonal element in $\Lambda(x)$ is positive (respectively, negative); 
$u\colon \mathbb R_+ \to \mathbb{R}^m$ denotes the control law; $B 
$ is an $n\times m$ real matrix; 
and $M 
$ is an $n \times n$ real matrix 
accounting for
the boundary conditions.

Our goal is to provide controllability results for System~\eqref{eq:hyperbolic} in the space $L^q([0, 1], \mathbb R^n)$ for $q \in [1, +\infty)$.
One classical strategy consists in reformulating the controllability problem into an observability inequality, as done in \cite{EngelKramar,Miller2004,Tucsnak_Weiss,Ramdani, OliveMinimal, Hu2021Minimal, HuMinimal, CoronNguyen2021, CoronNguyenoptimal}. In particular, \cite{OliveMinimal, Hu2021Minimal, HuMinimal, CoronNguyen2021, CoronNguyenoptimal} characterize the minimal time for null or exact controllability of \eqref{eq:hyperbolic} for some particular choices of matrices $M$ and $B$, but in the more general case where $D(x)$ is not necessarily diagonal. Another strategy consists in exploiting results about  
difference equations, as done in
\cite{bastin2016stability,CoNg,Chitour2016Stability,baratchart} for stability.
Here, 
as detailed in Section~\ref{sec:transformation}, we 
adopt the second strategy, applying
the method of characteristics to transform System~\eqref{eq:hyperbolic} into a linear difference 
equation of the form
\begin{equation}
\label{eq_generic_delay}
\begin{pmatrix}
 y_1(t) \\
 \vdots\\
 y_n(t)
 \end{pmatrix}=K \begin{pmatrix}
 y_1(t-\tau_1) \\
 \vdots\\
 y_n(t-\tau_n)
 \end{pmatrix}+Bu(t),\qquad \mbox{ 
 $t \ge 0$}, 
\end{equation}
where the state variable $(y_1(t),\dots,y_n(t))$ is related to the boundary values of $R(t,\cdot)$, the matrix $K$ is obtained from $M$, $\Lambda(\cdot)$, and $D(\cdot)$, and $\tau_1, \dotsc, \tau_n$ are obtained from $\Lambda(\cdot)$. We then prove (see Proposition~\ref{prop:control-equiv}) that approximate or exact controllability of \eqref{eq:hyperbolic} in $L^q([0, 1], \mathbb R^n)$ is equivalent to the same kind of controllability of \eqref{eq_generic_delay} in the state space $\Sigma^q := \prod_{i=1}^n L^q([-\tau_i,0], \allowbreak \mathbb{R})$.

ystem~\eqref{eq_generic_delay} can be seen as a special case of the more general linear difference equation
\begin{equation}
\label{system_lin_formel2}
 x(t)=\sum\limits_{j=1}^NA_jx(t-\Lambda_j)+\widetilde{B}u(t) , \qquad t \ge 0,
\end{equation}
where the positive real numbers $\Lambda_1 , \dotsc,\Lambda_N$ are the delays, $A_1,\dotsc,A_N$ are $d\times d$ real matrices, $\widetilde{B}$ is a $d \times m$ real matrix, and 
$t \mapsto u(t)$ is a $\mathbb{R}^m$-valued control law.
The exponential stability of systems of the form \eqref{system_lin_formel2} in the uncontrolled case $\widetilde{B}=0$ has been completely characterized (cf.~the monograph \cite{Hale}),
while controllability issues are less classical. 

In the paper \cite{ChitourHautus}, we gave necessary and/or sufficient conditions for the controllability of System~\eqref{system_lin_formel2}
in the state space $L^q([-\Lambda_{\rm max},0],\mathbb{R}^d)$, $q\in [1,+\infty)$, where $\Lambda_{\rm max}= \max_{j=1,\dots,N} \Lambda_j$.
  These results are expressed in the frequency domain and   
  are of Hautus type.
The results of \cite{ChitourHautus} provide necessary and/or sufficient conditions for controllability of System~\eqref{eq_generic_delay} in $L^q([-\tau_{\max},0],\mathbb{R}^n)$, where $\tau_{\max}=\max_{j=1,\dots,m}\tau_j$. 
Controllability of System~\eqref{eq_generic_delay} in $L^q([-\tau_{\max},0],\allowbreak \mathbb{R}^n)$ implies controllability in $\Sigma^q$ (which is equivalent to controllability of \eqref{eq:hyperbolic} in $L^q([0, 1], \mathbb R^n)$), and then the results of \cite{ChitourHautus} allow one to obtain sufficient controllability conditions for \eqref{eq:hyperbolic}. However, these results only yield conservative sufficient controllability conditions and, moreover, relying only on the results of \cite{ChitourHautus} does not allow one to derive necessary conditions, since controllability of System~\eqref{eq_generic_delay} in $\Sigma^q$ does not imply controllability in $L^q([-\tau_{\max},0],\mathbb{R}^n)$, as shown by the following example.

Consider \eqref{eq_generic_delay} in the case $n = 2$, $(\tau_1, \tau_2) = (1, \tau)$ with $\tau \in (0, 1)$, and
\[
K = \begin{pmatrix}
0 & 1 \\
1 & 0 \\
\end{pmatrix}, \qquad B = \begin{pmatrix} 0 \\ 1 \end{pmatrix}, 
\]
that is,
\begin{equation}
\label{eq:intro-example}
\left\{
\begin{aligned}
y_1(t) & = y_2(t - \tau), \\
y_2(t) & = y_1(t - 1) + u(t).
\end{aligned}
\right.
\end{equation}
It follows from \cite[Theorem~3.6]{ChitourHautus} that \eqref{eq:intro-example} is never approximately controllable in $L^q([-1, 0], \mathbb R^2)$. This can also be verified by a direct argument given next. For $t \geq 0$, let $y_{[t]}$ denote the history function of \eqref{eq:intro-example}, i.e., $y_{[t]}(s) = y(t + s)$ for $s \in [-1, 0]$. Then the first equation of \eqref{eq:intro-example} shows that, for all $t \geq 1 - \tau$, $y_{[t], 1}\rvert_{[\tau-1, 0]}(\cdot + \tau) = y_{[t], 2}\rvert_{[-1, -\tau]}$, and hence \eqref{eq:intro-example} cannot be exactly nor approximately controllable in $L^q([-1, 0], \mathbb R^2)$. 

On the other hand, let us check that \eqref{eq:intro-example} is exactly controllable in $\Sigma^q$ in time $1 + \tau$. Given an initial condition $(\phi_0, \psi_0)$ and a target state $(\phi_1, \psi_1)$, both in $\Sigma^q$, if $y_1(s) = \phi_0(s)$ for $s \in [-1, 0]$ and $y_2(s) = \psi_0(s)$ for $s \in [-\tau, 0]$, by choosing $u(t) = \phi_1(t - 1) - \phi_0(t - 1)$ for $t \in [0, 1]$ and $u(t) = \psi_1(t - 1 - \tau) - \psi_0(t - 1 - \tau)$ for $t \in [1, 1 + \tau]$, it is immediate to check that $y_1(1 + \tau + s) = \phi_1(s)$ for $s \in [-1, 0]$ and $y_2(1 + \tau + s) = \psi_1(s)$ for $s \in [-\tau, 0]$, yielding exact controllability in $\Sigma^q$ in time $1+\tau$.

At the light of the previous example, the results of \cite{ChitourHautus} are not sufficient for an extensive comprehension of the controllability of System~\eqref{eq:hyperbolic} and require additional work, which is the purpose of the present paper. Our first result is that approximate (respectively, exact) controllability from the origin of System~\eqref{eq_generic_delay} is equivalent to 
its approximate (respectively, exact) controllability in time $\tau_1+\dots+\tau_n$, thus providing an upper bound for the minimal time for controllability. This result is the counterpart of \cite[Theorem~4.7]{ChitourHautus} and requires a finer analysis, exploiting the particular structure of the difference equation \eqref{eq_generic_delay} with respect to the general one given in \eqref{system_lin_formel2}.
We then base our analysis on the realization theory for infinite dimensional linear systems developed by Y.~Yamamoto \cite{yamamoto1981realizationI,yamamoto1981realizationII,Yutaka_Yamamoto}.
The key point for applying this framework 
is the identification of an
input-output system equivalent to \eqref{eq_generic_delay} that is compatible with the functional space $\Sigma^q$.
As a consequence, we are able to obtain a Hautus-type necessary and sufficient condition for the approximate controllability 
in time $\tau_1+\dots+\tau_n$ of System~\eqref{eq:hyperbolic} in $L^q([0, 1], \mathbb R^n)$, $q\in [1,+\infty)$. 
We also derive a Hautus-type necessary and sufficient condition
for the exact controllability 
in time $\tau_1+\dots+\tau_n$ of System~\eqref{eq:hyperbolic} in $L^1([0, 1], \mathbb R^n)$. Note that Hautus criteria 
for controllability or observability
have been obtained in the literature for some classes of infinite-dimensional systems, such as systems with skew-adjoint generators in \cite{Miller2004} or exponentially stable systems in \cite{Jacob2009Hautus, Russell1994General}.

In the final part of the paper, 
as a special case of 
linear one-dimensional hyperbolic system, we consider flows in networks, namely dynamical graphs where the  edges are seen as transport equations with suitable linear static conditions at the vertices (see \cite{KramarSikolya2005,EngelKramar,Gantouh2020ApproximateCO,Sikolya}). 
Note that all these works rely on  observability inequalities and not on a difference equation
approach like ours. 
Applying our controllability results for System~\eqref{eq:hyperbolic} to flows in networks, we improve the existing results in the literature in two directions. On the one hand, we identify a topological obstruction for controllability, namely the graph must be the finite union of cycle graphs for (approximate or exact) controllability to hold true. On the other hand, when the graph is a finite union of cycle graphs, we also precise the Hautus criteria for controllability, noticing that it reduces to a Kalman criterion in the case where the delays are commensurable.

The paper is organized as follows.
 In Section~\ref{sec:notation}, we gather the main notation used throughout the paper. Section~\ref{sec:results}
provides first the complete description of System~\eqref{eq:hyperbolic} and of the one-to-one correspondence with System~\eqref{eq_generic_delay}, enabling one to define a general concept of solution for System~\eqref{eq:hyperbolic}. We then introduce a representation formula for solutions of System~\eqref{eq_generic_delay}, which is used in Section~\ref{sec:upper-bound-time} to derive the upper bound $\tau_1+\cdots+\tau_n$ for the minimal time of controllability. We finally give  several definitions of controllability both for System~\eqref{eq:hyperbolic} and System~\eqref{eq_generic_delay}, which are equivalent thanks to the transformation defined previously. Section~\ref{sec:control_delay_system} is devoted to the controllability criteria, 
obtained by applying Yamamoto's realization theory.
 Finally, Section~\ref{sec:appli} focuses on the case of flows in networks, introducing their general formulation, obtaining the topological obstruction for their controllability, and applying to them the general Hautus test.

\section{Notation}
\label{sec:notation}

In this paper, $\mathbb{N}$ represents the set of nonnegative integers. The sets of real, complex, nonnegative and nonpositive real numbers are denoted by $\mathbb{R}$, $\mathbb{C}$, $\mathbb{R}_+$ and $\mathbb{R}_-$ respectively. 
For $p \in \mathbb{C}$, $\Re(p)$ and $\Im(p)$ represent the real and imaginary parts of $p$. For two integers $a, b$ with $a \leq b$, we use $\llbracket a, b\rrbracket$ to denote the interval of integers between $a$ and $b$, i.e., $\llbracket a, b\rrbracket = [a, b] \cap \mathbb Z$, and, if $a>b$, $\llbracket a, b\rrbracket$ is understood to be the empty set.
Given $\ell=(\ell_1,\dots,\ell_n) \in \mathbb{N}^n$,  the length of the 
 $n$-tuple $\ell$ is denoted by $|\ell|$ and is equal to $\ell_1+ \dots + \ell_n$.

Given two positive integers $n$ and $m$, $\mathcal{M}_{n,m}(\mathbb{K})$ is the set of $n \times 
m$ matrices with coefficients in $\mathbb{K}=$ $\mathbb{R}$ or $ \mathbb{C}$.  For $A \in 
\mathcal{M}_{n,m}(\mathbb{K})$, we denote by $A^\ast$ its conjugate transpose matrix. 
We use 
$\|\cdot\|$ to denote a norm for every finite-dimensional space (over $\mathbb{K}$).

Elements 
$x\in \mathbb{K}^n$ are considered as column vectors, hence the transposition $x^T$ of $x$ is a row vector. 
Given $x,y\in \mathbb{K}^n$, 
we denote $x\cdot y=x^T y$. 
For $x \in  \mathbb{K}^n$, we use $\diag\,(x_1,\dots,x_n)$ to denote the diagonal matrix whose diagonal entries are the components of $x$, and we extend this notation for block-diagonal matrices when $x_1, \dotsc, x_n$ are square matrices. The identity matrix in $\mathcal{M}_{n,n}(\mathbb{K})$ is denoted by $I_n$. For $M \in \mathcal{M}_{n,m}(\mathbb{K})$, $\rank M$ denotes the 
rank of $M$, i.e., the dimension of the linear space over $\mathbb{K}$ spanned by its columns. 
Given a positive integer $k$, $A \in \mathcal{M}_{n,m}(\mathbb{K})$, and $B\in  \mathcal{M}_{n,k}(\mathbb{K})$, the bracket $\left[A,B\right]$ denotes the juxtaposition of the two matrices, which hence belongs to $\mathcal{M}_{n,m+k}(\mathbb{K})$.

Given $q \in [1,+\infty]$, we endow  $L^q_{\rm loc}\left(\mathbb{R}_+,\mathbb{R}^k\right)$
with the topology induced by the  semi-norms
\[
\|\phi \|_{[0,a],\,q}:=
\begin{cases}
    \left(\int_0^a \vertii{\phi(t)}^q dt  \right)^{1/q},&\mbox{ if }q\in [1,+\infty),\\
    \esssup \{\|\phi(t)\| \mid t\in [0,a]\},&\mbox{ if }q=+\infty,
\end{cases}
\qquad \text{for } a> 0.
\]
Similarly, given
an interval $I$ of $\mathbb{R}$, $L^q(I,\mathbb{R}^k)$ is 
  endowed of the norm $\| \cdot \|_{I,q}$. 
The space $\Sigma^q=\prod_{j=1}^n L^q\left([-\tau_i,0],\mathbb{R}\right)$, with $\tau_1,\dots,\tau_n>0$,
is endowed with the product norm, denoted by $\|\cdot \|_{\Sigma^q}$.

  For a linear operator $L$, 
 we denote by $\Ran L$ its range and by $ \overline{\Ran L}$ the closure of its 
 range. Similarly, if $F$ is a matrix-valued holomorphic function,  we use 
 $F(\mathbb{C})$ and $\overline{F(\mathbb{C})}$ to denote its image and the closure of its image, respectively.

We define $M_+(\mathbb{R})$ to be the space of Radon measures on $\mathbb{R}$ whose support left-bounded, and we let $M(\mathbb{R}_{-})$ denote the subset of $M_+(\mathbb R)$ made of the Radon measures whose support is  compact and contained in $\mathbb{R}_{-}$. We denote by $*$ the convolution in $M_+(\mathbb{R})$,
which is then seen as a convolution algebra. Given
 a Radon measure
$\mu \in M_+(\mathbb{R})$, we denote by $\widehat{\mu}(p)$ the Laplace transform of $\mu$ at the frequency $p\in \mathbb{C}$, that is, 
\begin{equation}
\label{laplace_transform_radon_measure}
\widehat{\mu}(p)=\int_{-\infty}^{+\infty} d\mu(t)e^{-pt},
\end{equation}
 provided that the previous integral 
 exists.
 The Dirac distribution at $x \in \mathbb{R}$ is denoted by  $\delta_x \in M_+(\mathbb{R})$.
Notice that the Laplace transform of
$\delta_x$ is the holomorphic map 
\begin{equation}
\label{laplace_transform_dirac}
\widehat{\delta}_x(p)=e^{-p x},\qquad p \in \mathbb{C}.
\end{equation}

\section{Controllability problems for 
linear one-dimensional hyperbolic systems}
\label{sec:results}\label{sec:problems}

We consider System~\eqref{eq:hyperbolic}
where

\begin{itemize}[label={--}]

\item
$\Lambda(x) = \diag\{\lambda_1(x), \dotsc, \lambda_n(x)\}$ and there exists $\tilde n \in \{0, \dotsc, n\}$ such that
\[
\lambda_i(x) >0 > \lambda_j(x) \qquad \text{ for every } i \in  \{1, \dotsc, \tilde n\}
,\ j \in\{\tilde n + 1, \dotsc, n\},\ x \in [0, 1],
\]
and, in addition, $\lambda_i$ and $\frac{1}{\lambda_i}$ belong to $L^\infty((0, 1), \mathbb R)$ for every $i \in \{1, \dotsc, n\}$;

\item
the components of the solution $R$ are split into those corresponding to positive and negative velocities, i.e., 
\[
R= \begin{pmatrix}
 R^+ \\
 R^-
 \end{pmatrix}\quad\mbox{with}\ 
 \left\{ \begin{array}{rcl}
    R^+ &=& (R_1, \dots, R_{\tilde n})^T, \\
    R^- &=& (R_{\tilde n+1}, \dots, R_n)^T.
  \end{array} \right.
\]

\item
$D(x) = \diag \{d_1(x), \dotsc, d_n(x)\}$ and $d_i \in L^1((0, 1), \mathbb R)$ for every $i \in \{1, \dotsc, n\}$.

\end{itemize}

Notice that the choice of taking every transport equation on the interval $[0,1]$ corresponds to a choice of normalization. One could equivalently consider each transport equation on an arbitrary space interval and recover \eqref{eq:hyperbolic} by  rescaling the speeds $\lambda_1,\dots,\lambda_n$.

\subsection{Transformation into linear difference equations and well-po\-sed\-ness}\label{sec:transformation}

In this section we start by providing a definition of solution to System~\eqref{eq:hyperbolic}, directly inspired by the method of characteristics and which 
does not require any a priori regularity on the solution.
The definition below is an immediate generalization of  \cite[Definition~4.1]{Chitour2016Stability}.

\begin{definition}
\label{def:sol-EDP}
Let $T > 0$, $u\colon [0, T] \to \mathbb R^m$, and $\bar R\colon [0, 1] \to \mathbb R^n$. We say that $R \colon [0, T] \times [0, 1] \to \mathbb R^n$ is a \emph{solution} of \eqref{eq:hyperbolic} in $[0, T]$ with initial condition $\bar R$ and control $u$ if $R(0, x) = \bar R(x)$ for every $x \in [0, 1]$, \eqref{eq_bord} is satisfied for every $t \geq 0$, and, for every $i \in \llbracket 1, n\rrbracket$, $t \in [0, T]$, and $x \in [0, 1]$, we have
\begin{equation}
\label{eq:weak-sol}
R_i\left(t + \int_x^{x + h} \frac{d\xi}{\lambda_i(\xi)}, x + h\right) = e^{-\int_{x}^{x + h} \frac{d_i(\xi)}{\lambda_i(\xi)} d\xi} R_i(t, x)
\end{equation}
for every $h \in \mathbb R$ such that $t + \int_x^{x + h} \frac{d\xi}{\lambda_i(\xi)} \in [0, T]$ and $x + h \in [0, 1]$.
\end{definition}

\begin{remark}
A similar concept of solution for systems under a form more general than \eqref{eq:hyperbolic} has been given in \cite[Definition~3.1]{CoronNguyenoptimal}, where it is referred to as \emph{broad solution}. 
\end{remark}

\begin{remark}
The motivation for the notion of solution provided in Definition~\ref{def:sol-EDP} is that \eqref{eq:weak-sol} characterizes classical solutions of \eqref{eq_hyp}. More precisely, given any $T>0$, if $R \in C^1((0, T) \times (0, 1), \mathbb R^n) \cap C^0([0, T] \times [0, 1], \mathbb R^n)$ satisfies \eqref{eq:weak-sol} for every $i \in \llbracket 1, n\rrbracket$, $t \in [0, T]$, and $x \in [0, 1]$, then, by differentiating \eqref{eq:weak-sol} with respect to $h$ at $h = 0$, we deduce that \eqref{eq_hyp} is satisfied on $(0, T) \times (0, 1)$. Conversely, if $R \in C^1((0, T) \times (0, 1), \mathbb R^n) \cap C^0([0, T] \times [0, 1], \mathbb R^n)$ satisfies \eqref{eq_hyp} on $(0, T) \times (0, 1)$, an application of the method of characteristics shows that \eqref{eq:weak-sol} is satisfied for every $i \in \llbracket 1, n\rrbracket$, $t \in [0, T]$, and $x \in [0, 1]$.
\end{remark}

We now relate solutions of \eqref{eq:hyperbolic} with those of the linear difference equation \eqref{eq_generic_delay}. For that purpose, set
\[
K = M \diag\left\{e^{-\int_0^1 \frac{d_1(x)}{\abs{\lambda_1(x)}} d x}, \dotsc, e^{-\int_0^1 \frac{d_n(x)}{\abs{\lambda_n(x)}} d x} \right\},
\]
and, for $i \in \llbracket 1, n\rrbracket$, define $\psi_i\colon [0, 1] \to \mathbb R$ by
\[
\psi_i(x) = \begin{dcases*}
\int_0^x \frac{d\xi}{\lambda_i(\xi)} & if $i \in \llbracket 1, \tilde n\rrbracket$, \\
\int_x^1 \frac{d\xi}{\abs{\lambda_i(\xi)}} & if $i \in \llbracket \tilde n + 1, n \rrbracket$.
\end{dcases*}
\]
Note that $\psi\colon [0,1] \to [0,\tau_i]$ is an homeomorphism for $i\in \llbracket 1,n\rrbracket$, where one sets
\begin{equation}\label{eq:taui}
\tau_i = \int_0^1 \frac{dx}{\abs{\lambda_i(x)}},\qquad i \in \llbracket 1, n \rrbracket.
\end{equation}

We use here the definition of solution of \eqref{eq_generic_delay} from \cite[Definition~2.1]{Chitour2020Approximate} and \cite[Definition~2.1]{Mazanti2017Relative}, which, similarly to Definition~\ref{def:sol-EDP}, does not require any a priori regularity on the solution. More precisely, the family $y_i\colon [-\tau_i, T] \to \mathbb R$, $i \in \llbracket 1, n\rrbracket$, is said to be a \emph{solution of \eqref{eq_generic_delay}} in $[0, T]$ with a given control $u\colon [0, T] \to \mathbb R^m$ if \eqref{eq_generic_delay} is satisfied for every $t \in [0, T]$. In that case, the family $\bar y_{i}\colon [-\tau_i, 0) \to \mathbb R$, $i \in \llbracket 1, n\rrbracket$, where $\bar y_{i}$ is the restriction of $y_i$ to $[-\tau_i, 0)$, is said to be the \emph{initial condition} of the solution $(y_i)_{i \in \llbracket 1, n\rrbracket}$.

We thus have the following result.

\begin{proposition}
\label{prop:equiv-edp-diff}
Let $R \colon [0, T] \times [0, 1] \to \mathbb R^n$ be a solution of \eqref{eq:hyperbolic} in $[0, T]$ with initial condition $\bar R$ and control $u\colon [0, T] \to \mathbb R^m$. For $i \in \llbracket 1, \tilde n\rrbracket$, define $y_i\colon [-\tau_i, T] \to \mathbb R$ by
\begin{subequations}
\label{eq:y-from-R}
\begin{equation}
y_i(t) = \begin{dcases*}
R_i(t, 0) & if $t \geq 0$, \\
e^{\int_0^{\psi_i^{-1}(-t)} \frac{d_i(\xi)}{\lambda_i(\xi)} d\xi} \bar R_i(\psi_i^{-1}(-t)) & if $-\tau_i \leq t < 0$,
\end{dcases*}
\end{equation}
and, for $i \in \llbracket \tilde n + 1, n\rrbracket$, define $y_i\colon [-\tau_i, T] \to \mathbb R$ by
\begin{equation}
y_i(t) = \begin{dcases*}
R_i(t, 1) & if $t \geq 0$, \\
e^{\int_{\psi_i^{-1}(-t)}^1 \frac{d_i(\xi)}{\abs{\lambda_i(\xi)}} d\xi} \bar R_i(\psi_i^{-1}(-t)) & if $-\tau_i \leq t < 0$.
\end{dcases*}
\end{equation}
\end{subequations}
Then $(y_1,\dots,y_n)$ is a solution of \eqref{eq_generic_delay} in $[0, T]$ with control $u$.

Conversely, let 
the family $y_i\colon [-\tau_i, T] \to \mathbb R$, $i \in \llbracket 1, n\rrbracket$, be a solution of \eqref{eq_generic_delay} in $[0, T]$ with control $u\colon [0, T] \to \mathbb R^m$. For  $i \in \llbracket 1, \tilde n\rrbracket$, define $R_i\colon [0, T] \times [0,1]\to \mathbb R$ by
\begin{subequations}
\label{eq:R-from-y}
\begin{equation}
R_i(t, x) = e^{-\int_0^x \frac{d_i(\xi)}{\lambda_i(\xi)} d\xi} y_i(t - \psi_i(x)),
\end{equation}
and, for $i \in \llbracket \tilde n + 1, n\rrbracket$, define $R_i\colon [0, T] \times [0,1] \to \mathbb R$ by
\begin{equation}
R_i(t, x) = e^{-\int_x^1 \frac{d_i(\xi)}{\abs{\lambda_i(\xi)}} d\xi} y_i(t - \psi_i(x)).
\end{equation}
\end{subequations}
Then $R$ is a solution of \eqref{eq:hyperbolic} in $[0, T]$ 
with control $u$.
\end{proposition}

Proposition~\ref{prop:equiv-edp-diff} extends \cite[Proposition~4.2]{Chitour2016Stability}, which corresponds to the particular case $u = 0$, $\Lambda(\cdot)$ constant, and $D = 0$. The proof of the latter can be immediately adapted to the present setting, and we omit it here for conciseness.

\begin{remark}
Under additional regularity assumptions on $\Lambda$ and $D$, other works have also addressed the equivalence between weak notions of solution of \eqref{eq:hyperbolic} and \eqref{eq_generic_delay} (see, e.g., \cite{gugat-contamination}).
\end{remark}

It is known (see, e.g., \cite[Proposition~2.2]{Chitour2020Approximate} or \cite[Proposition~2.2]{Mazanti2017Relative}) that, for every $T > 0$, every family $\bar y_{i}\colon [-\tau_i, 0) \to \mathbb R$, $i \in \llbracket 1, n\rrbracket$, and every $u \colon [0, T] \to \mathbb R^m$, there exists a unique solution $y_i\colon [-\tau_i, T] \to \mathbb R$, $i \in \llbracket 1, n\rrbracket$, of \eqref{eq_generic_delay} with initial condition $(\bar y_{i})_{i \in \llbracket 1, n\rrbracket}$ and control $u$. The next result is an immediate consequence of this fact and the correspondence between solutions of \eqref{eq:hyperbolic} and those of \eqref{eq_generic_delay} from Proposition~\ref{prop:equiv-edp-diff}.

\begin{proposition}
\label{prop:exist}
Let $T > 0$, $\bar R\colon [0, 1] \to \mathbb R^n$, and $u\colon [0, T] \to \mathbb R^m$. Then \eqref{eq:hyperbolic} admits a unique solution $R\colon [0, T] \times [0, 1] \to \mathbb R^n$ in $[0, T]$ with initial condition $\bar R$ and control $u$.
\end{proposition}

Our subsequent goal is to provide controllability results in $L^q$-type spaces for $q\in [1, +\infty]$, and hence one needs to address the correspondence between solutions of \eqref{eq:hyperbolic} and \eqref{eq_generic_delay} from Proposition~\ref{prop:equiv-edp-diff} also in an $L^q$ framework. For that purpose, we define, for $q \in [1, +\infty]$, the space
\[
\Sigma^q = \prod_{i=1}^n L^q((-\tau_i, 0), \mathbb R)
\]
and the map $\mathcal T_q\colon \Sigma^q \to L^q((0, 1), \mathbb R^n)$ which associates, with every $y\in \Sigma^q$, the element $r\in L^q((0, 1), \mathbb R^n)$
given, for $x\in(0,1)$, by
\[
r_i(x) = \begin{dcases*}
e^{-\int_0^x \frac{d_i(\xi)}{\lambda_i(\xi)} d\xi} y_i(-\psi_i(x)), & if $i \in \llbracket 1, \tilde n\rrbracket$, \\
e^{-\int_x^1 \frac{d_i(\xi)}{\abs{\lambda_i(\xi)}} d\xi} y_i(-\psi_i(x)), & if $i \in \llbracket \tilde n + 1, n\rrbracket$.
\end{dcases*}
\]
Note that, since $\frac{d_i}{\lambda_i}\in L^1((0, 1), \mathbb R)$, it follows from \cite[Theorem~3.9]{Evans2015Measure} that $r_i \in L^q((0, 1), \mathbb R)$ and $\mathcal T_q$ is a well-defined bounded linear operator. In addition, since $\frac{1}{\lambda_i} \in L^\infty((0, 1), \mathbb R)$,  once again using \cite[Theorem~3.9]{Evans2015Measure}, we deduce that the operator $\mathcal T_q$ is invertible, and its inverse $\mathcal T_q^{-1}\colon L^q((0, 1), \mathbb R^n) \to \Sigma^q$ is the bounded linear operator associating, with every $r \in L^q((0, 1), \mathbb R^n)$, the element $y \in \Sigma^q$ given, for $i \in \llbracket 1, n\rrbracket$ and $t \in (-\tau_i, 0)$, by
\[
y_i(t) = \begin{dcases*}
e^{\int_0^{\psi_i^{-1}(-t)} \frac{d_i(\xi)}{\lambda_i(\xi)} d\xi} r_{i}(\psi_i^{-1}(-t)) & if $i \in \llbracket 1, \tilde n\rrbracket$, \\
e^{\int_{\psi_i^{-1}(-t)}^1 \frac{d_i(\xi)}{\abs{\lambda_i(\xi)}} d\xi} r_{i}(\psi_i^{-1}(-t)) & if $i \in \llbracket \tilde n + 1, n\rrbracket$.
\end{dcases*}
\]

Given a family $y_i\colon [-\tau_i, T] \to \mathbb R$, $i \in \llbracket 1, n\rrbracket$, and $t \in [0, T]$, we use $y_{[t]}$ to denote the family of functions $y_{[t], i}\colon [-\tau_i, 0) \to \mathbb R$, $i \in \llbracket 1, n\rrbracket$, defined by 
\[y_{[t], i}(s) = y_i(t + s),\qquad
s \in [-\tau_i, 0).\]

Using the maps $\mathcal T_q$ and $\mathcal T_q^{-1}$, we can now state an analogue of Proposition~\ref{prop:equiv-edp-diff} for $L^q$ solutions of \eqref{eq:hyperbolic} and \eqref{eq_generic_delay}.

\begin{proposition}
\label{prop:equiv-Lq}
Let $q\in [1, +\infty]$. Consider 
a solution $R \colon [0, T] \times [0, 1] \to \mathbb R^n$  of \eqref{eq:hyperbolic} in $[0, T]$ with initial condition $\bar R$ and control $u\colon [0, T] \to \mathbb R^m$, and assume that $R(t, \cdot) \in L^q((0, 1), \mathbb R^n)$ for every $t \in [0, T]$. Let $y_i\colon [-\tau_i, T] \to \mathbb R$, $i \in \llbracket 1, n\rrbracket$, be the solution of \eqref{eq_generic_delay} in $[0, T]$ with control $u$ defined by \eqref{eq:y-from-R}. Then, for every $t \in [0, T]$, we have
\[
y_{[t]} \in \Sigma^q \qquad \text{ and } \qquad y_{[t]} = \mathcal T_q^{-1} R(t, \cdot).
\]

Conversely, let the family $y_i\colon [-\tau_i, T] \to \mathbb R$, $i \in \llbracket 1, n\rrbracket$, be a solution of \eqref{eq_generic_delay} in $[0, T]$ with control $u\colon [0, T] \to \mathbb R^m$, and assume that $y_{[t]} \in \Sigma^q$ for every $t \in [0, T]$. Let $R\colon [0, T] \times [0,1]\to \mathbb R^n$ be the solution of \eqref{eq:hyperbolic} in $[0, T]$ with control $u$ defined by \eqref{eq:R-from-y}. Then, for every $t \in [0, T]$, we have
\[
R(t, \cdot) \in L^q((0, 1), \mathbb R^n) \qquad \text{ and } \qquad R(t, \cdot) = \mathcal T_q y_{[t]}.
\]
\end{proposition}

\begin{proof}
The first part of the proposition follows by rewriting $y_{[t], i}(s) = y_i(t + s)$ in terms of $R(t, \cdot)$ and by using \eqref{eq:weak-sol} in the expression of $y_{[t], i}(s)$ given by \eqref{eq:y-from-R}. The second part of the proposition is an immediate consequence of \eqref{eq:R-from-y}.
\end{proof}

Concerning existence of solutions in an $L^q$ framework, \cite[Remark~2.3]{Mazanti2017Relative} states that, for every $T > 0$, every family $\bar y_{i} \in L^q((-\tau_i, 0), \mathbb R)$, $i \in \llbracket 1, n\rrbracket$, and every $u \in L^q((0, T), \mathbb R^m)$, the unique solution $y_i\colon [-\tau_i, T] \to \mathbb R$, $i \in \llbracket 1, n\rrbracket$, of \eqref{eq_generic_delay} with initial condition $(\bar y_{i})_{i \in \llbracket 1, n\rrbracket}$ and control $u$ satisfies $y_i \in L^q((-\tau_i, T), \mathbb R)$ for every $i \in \llbracket 1, n\rrbracket$. Relying on Proposition~\ref{prop:equiv-Lq}, we obtain as immediate consequence the following result.

\begin{proposition}\label{prop:ex-sol-Lq}
Let $q \in [1, +\infty]$, $T > 0$, $\bar R \in L^q((0, 1), \mathbb R^n)$, and $u \in L^q((0, T), \mathbb R^m)$. Then \eqref{eq:hyperbolic} admits a unique solution $R\colon [0, T] \times [0, 1] \to \mathbb R^n$ in $[0, T]$ with initial condition $\bar R$ and control $u$, which satisfies $R(t, \cdot) \in L^q((0, 1),\allowbreak \mathbb R^n)$ for every $t \in [0, T]$.
\end{proposition}

\begin{remark}
Definition~\ref{def:sol-EDP} still makes sense and Propositions~\ref{prop:equiv-edp-diff} and~\ref{prop:exist} remain true under the weaker assumption on $\Lambda$ that $\frac{1}{\lambda_i} \in L^1((0, 1), \mathbb R)$ for every $i \in \llbracket 1, n\rrbracket$. 
The stronger assumption that $\lambda_i$ and $\frac{1}{\lambda_i}$ belong to $L^\infty((0, 1), \mathbb R)$ are needed only to deal with solutions in $L^q$
in order for the operators $\mathcal T_q$ and $\mathcal T_q^{-1}$ to be well-defined, and, in the case $q = +\infty$, they can be relaxed to $\lambda_i$ and $\frac{1}{\lambda_i}$ belonging to $L^1((0, 1), \mathbb R)$.
\end{remark}

\subsection{Recursive representation formula}

Motivated by the links between solutions of \eqref{eq:hyperbolic} and \eqref{eq_generic_delay} established in the previous section, we now provide a representation formula for the solutions of \eqref{eq_generic_delay}, which is an immediate adaptation of the representation formula for solutions of linear difference equations of the form \eqref{system_lin_formel2} from \cite{Chitour2020Approximate, Mazanti2017Relative}. For that purpose, we first provide the following definition.

\begin{definition}
\label{def_Xi_n}
The family of matrices $\Xi_\ell \in \mathcal{M}_{n,n}(\mathbb{R})$, $\ell \in \mathbb{Z}^n$, is defined recursively as
\begin{equation}
 \Xi_\ell= \begin{dcases*}
     0 & if $\ell \in \mathbb{Z}^n \setminus \mathbb{N}^n$,\\
      I_n & if  $\ell=0$,\\
 \sum_{k=1}^n K e_{{k}} e_{{k}}^T \Xi_{\ell - e_k} & if $\ell \in \mathbb{N}^n$ and $|\ell|>0$,
    \end{dcases*}   
\end{equation}
where $e_1, \dotsc, e_n$ denotes the canonical basis of $\mathbb{R}^n$.

\end{definition}

\begin{remark}
We remark that $K e_k e_k^T$ represents the $n \times n$ matrix whose $k$th column coincides with that of  $K$ and such that all other elements are zero. Its appearance is motivated by the identity\footnote{Strictly speaking, this identity makes sense only for $t$ large enough (namely, $t\ge \tau_{\max}$) since, for small $t>0$, some components of $y(t - \tau_k)$ may fail to be defined.}
\[ K \begin{pmatrix}
 y_1(t-\tau_1) \\
 \vdots\\
 y_n(t-\tau_n)
 \end{pmatrix}=\sum_{k=1}^n K e_k e_k^T y(t-\tau_k),\]
where $y(\cdot) = (y_1(\cdot), \dotsc, y_n(\cdot))$.
\end{remark}

Using the matrices $\Xi_\ell$, $\ell \in \mathbb Z^n$, defined above, we provide the following definitions of the flow and endpoint operators for \eqref{eq_generic_delay}.

\begin{definition}
\label{def_Upsilon_E}
Let $T \geq 0$, $q \in [1,+\infty]$, and denote by $\tau$ the vector $(\tau_1, \dotsc, \tau_n)$.

\begin{enumerate}
\item The \emph{flow operator} $\Upsilon_q(T)\colon \Sigma^q \to \Sigma^q$ is defined, for $\phi \in \Sigma^q$ and $i \in \llbracket 1, n\rrbracket$,  by
\[
\left(\Upsilon_q(T)\phi \right)_i(s) = e_i^T \sum_{\substack{(\ell, j) \in \mathbb{N}^n \times \llbracket 1,n\rrbracket \\ -\tau_j \le T+s-\tau \cdot \ell <0}}  \Xi_{\ell - e_j} K e_j \phi_j(T + s - \tau \cdot \ell),\qquad 
s \in [-\tau_i,0].
\]
\item The \emph{endpoint operator} $E_q(T)\colon L^q([0,T],\mathbb{R}^m) \to \Sigma^q$ is defined, for $u \in L^q([0,T],\mathbb{R}^m)$ and $i \in \llbracket 1, n \rrbracket$, by 
\[
\left(E_q(T)u \right)_i(t)= e_i^T\sum_{\substack{\ell \in \mathbb{N}^n \\ \tau \cdot \ell \le T+t}} \Xi_\ell B u(T+t-\tau \cdot \ell), \qquad 
t\in [-\tau_i,0].
\]
\end{enumerate}
\end{definition}

It follows immediately from their expressions that $\Upsilon_q(T)$ and $E_q(T)$ are bounded linear operators. The main motivation for introducing them is that they can be used to express the solutions of \eqref{eq_generic_delay}, as stated in the following result, which is based on the recursive application of \eqref{eq_generic_delay} in order to express $y_1(t), \dotsc, y_n(t)$ in terms of the initial condition and the control only (cf.~\cite[Proposition~2.4]{Chitour2020Approximate} or \cite[Proposition~2.7]{Mazanti2017Relative}).

\begin{proposition}
\label{prop_var_const_formula}
For $T \geq 0$, $q\in [1,+\infty]$, $u \in L^q([0,T],\allowbreak \mathbb{R}^m)$, and $\phi \in \Sigma^q$, the unique solution $y$ of \eqref{eq_generic_delay} with initial condition $\phi$ and control $u$ satisfies, for every $t \in [0,T]$,
\begin{equation}
\label{eq:variation_constant}
y_{[t]}=\Upsilon_q(t)\phi+E_q(t)u.
\end{equation}
\end{proposition}

\begin{remark}
\label{remk:RanETnondecre}
As a consequence of \eqref{eq:variation_constant}, the operator $E_q(T)$ associates with a control $u$ the state at time $T$ of the corresponding trajectory starting at $0\in \Sigma^q$ at time $0$.
In particular, $T\mapsto \Ran E_q(T)$ is non-decreasing, since $E_q(T_1)u_1=E_q(T_1+T_2)u_2$ whenever $u_2$ is equal to zero on $[0,T_2)$ and to $u_1(\cdot - T_2)$ on $[T_2,T_1+T_2]$.
\end{remark}

In order to provide controllability characterizations of System~\eqref{eq_generic_delay}, one may consider the dual operator $E_q(T)^{\ast}$ of $E_q(T)$, whose explicit expression, in the case $q<\infty$, is given next (see \cite[Lemma~2.9]{Chitour2020Approximate}).

\begin{proposition}
\label{prop_operator_dual_E}
Let $T \geq 0$, $q \in [1,+\infty)$, and denote by $q'$ the conjugate exponent of $q$.
The dual  $E_q(T)^*$ of the operator $E_q(T)$ is the linear operator from $\Sigma^{q'}$ into $L^{q'}([0,T],\mathbb{R}^m)$ given by
\begin{equation}
\label{eq:op_E^*}
\left(E_q(T)^*y \right)_i(t)= e_i^T
\sum_{\substack{(\ell, j) \in \mathbb{N}^n \times \llbracket 1, n\rrbracket \\ -\tau_{j} \le t-T+\tau \cdot \ell<0}} B^* \Xi_\ell^* e_j y_j(t-T+\tau \cdot \ell),\qquad y \in \Sigma^{q'},\;
t \in [0,T].
\end{equation}
\end{proposition}

We also prove a useful property on the coefficients $\Xi_{\ell}$, $\ell \in \mathbb Z^n$, which is an improved version of \cite[Lemma~4.6]{ChitourHautus}.

\begin{lemma}
\label{lem:1}
Let $\Xi_\ell$, $\ell\in \mathbb{Z}^n$, be the matrices introduced in Definition~\ref{def_Xi_n}. 
There exist real coefficients $\alpha_k$ for $k \in \{0, 1\}^n$
such that for every $j\in \llbracket 1, n\rrbracket$ and 
$\ell \in \{\ell' \in \mathbb N^n \suchthat \max_{i \in \llbracket 1, n\rrbracket} \ell'_i \geq 2 \text{ or } \ell'_j = 1\}$,
we have
\begin{equation}
\label{eq:lem1:1}
e_j^T\Xi_\ell = -\sum_{k \in \{0, 1\}^n \setminus \{(0,\dots,0)\}}\alpha_k e_j^T\Xi_{\ell-k}.
\end{equation}
\end{lemma}

\begin{proof}
For $t=(t_1,\dots,t_n) \in \mathbb{R}^n$, set
\[
K(t) = K \diag(t_1, \dotsc, t_n) = t_1 K e_1 e_1^T + \dotsb + t_n K e_n e_n^T.
\]
By Definition~\ref{def_Xi_n} and an immediate induction argument, one deduces that 
\begin{equation}
\label{eq:lem1:1.1}
K(t)^j=\sum_{\substack{\ell \in \mathbb{N}^n \\ \abs{\ell} = j}} \Xi_\ell t^\ell,\qquad j\in \mathbb{N},\quad t \in \mathbb{R}^n,
\end{equation}
where $t^\ell := t_1^{\ell_1} \dotsm t_n^{\ell_n}$.  Indeed, \eqref{eq:lem1:1.1} is verified for $j = 0$ and, if $j \in \mathbb N$ is such that the equality in \eqref{eq:lem1:1.1} holds true for every $t \in \mathbb R^n$, then
\[
K(t)^{j+1} = \sum_{k=1}^n \sum_{\substack{\ell \in \mathbb{N}^n \\ \abs{\ell} = j}} t_k t^\ell K e_k e_k^T \Xi_\ell = \sum_{k=1}^n \sum_{\substack{\ell^\prime \in \mathbb{N}^n \\ \abs{\ell^\prime} = j + 1}} t^{\ell^\prime} K e_k e_k^T \Xi_{\ell^\prime - e_k} = \sum_{\substack{\ell^\prime \in \mathbb{N}^n \\ \abs{\ell^\prime} = j + 1}}  t^{\ell^\prime} \Xi_{\ell^\prime}.
\]
Notice that, by a similar induction argument and Definition~\ref{def_Xi_n}, one gets that $\|\Xi_\ell\|\le (n\|K\|)^{|\ell|}$ for every $\ell\in \mathbb N^n$. Using Neumann series, we deduce from Equation~\eqref{eq:lem1:1.1} 
that, for $t$  small enough,
\begin{align}
\big(I_n-K(t)\big)^{-1}&=\sum_{j\in \mathbb{N}} K(t)^j
=\sum_{\ell\in \mathbb{N}^n} \Xi_{\ell} t^\ell.\label{eq:lem1:2}
\end{align}
Notice that $P(t)=\det\big(I_n-K(t)\big)$ is a multivariate polynomial of degree $n$. Since $t_j$ appears only in the $j$th column of $I_n-K(t)$, the variable $t_j$ does not appear in any $(i,j)$-minor of $I_n-K(t)$ for $i\in \llbracket1, n\rrbracket$. Hence, we have
\begin{equation}
\label{eq:lem1:3}
P(t)=\sum_{k \in \{0, 1\}^n} \alpha_k t^k,
\end{equation}
for some real numbers 
$\alpha_k$ defined for $k \in \{0, 1\}^n$, with $\alpha_0=1$ and $\alpha_{(1, \dotsc, 1)} = (-1)^n \det K$.

Let $\operatorname{Adj}\big(I_n-K(t)\big)$ be the adjugate matrix of $I_n-K(t)$. Reasoning as for $P(t)$, we deduce that there exist matrices $M_\ell \in \mathcal{M}_{n, n}(\mathbb{R})$ for $\ell \in \{0, 1\}^n$, with $M_0 = \operatorname{Adj}(I_n) = I_n$, such that
\begin{equation}
\label{eq:lem1:4}
\operatorname{Adj}\big(I_n-K(t)\big)=\sum_{\ell \in \{0, 1\}^n}  M_{\ell} t^\ell.
\end{equation}
Moreover, for every $j \in \llbracket 1, n\rrbracket$, the $j$th row of $\operatorname{Adj}\big(I_n-K(t)\big)$ does not depend on $t_j$, and thus 
\begin{equation}\label{eq:ejTMell}
e_j^T M_\ell = 0\qquad \mbox{if}\quad \ell_j = 1.
\end{equation}
In particular, $M_{(1, \dotsc, 1)} = 0$. On the other hand, Equations~\eqref{eq:lem1:2} and \eqref{eq:lem1:3} lead to
\begin{align}
\operatorname{Adj}\big(I_n-A(t)\big) = P(t)\big(I_n-A(t)\big)^{-1} 
= \sum_{k \in \{0, 1\}^n} \sum_{\ell \in \mathbb{N}^n} \alpha_{k} \Xi_{\ell} t^{k + \ell} 
& = \sum_{\ell \in \mathbb{N}^n} \sum_{k \in \{0, 1\}^n} \alpha_{k} \Xi_{\ell - k} t^{\ell}, \label{eq:lem1:6}
\end{align}
where we also use the fact that $\Xi_k = 0$ if $k \in \mathbb Z^n \setminus \mathbb N^n$. Comparing Equations~\eqref{eq:lem1:4} and \eqref{eq:lem1:6}, we deduce that, for $\ell \in \mathbb{N}^n$, 
\begin{equation*}
\sum_{k \in \{0, 1\}^n} \alpha_{k} \Xi_{\ell - k} =
\begin{dcases*}
M_\ell & if $\ell \in \{0, 1\}^n$, \\
0 & otherwise.
\end{dcases*}
\end{equation*}
Since $\alpha_0 = 1$, we have
\begin{equation*}
e_j^T \Xi_\ell=
\begin{dcases*}e_j^T M_\ell-
\sum_{k \in \{0, 1\}^n\setminus\{(0,\dotsc,0)} \alpha_{k} e_j^T\Xi_{\ell - k}
 & if $\ell \in \{0, 1\}^n$, \\
-\sum_{k \in \{0, 1\}^n\setminus\{(0,\dotsc,0)} \alpha_{k} e_j^T\Xi_{\ell - k} & otherwise.
\end{dcases*}
\end{equation*}
The conclusion follows from \eqref{eq:ejTMell}.
\end{proof}

\subsection{Controllability notions}

Let us now introduce the controllability notions for Systems~\eqref{eq:hyperbolic} and \eqref{eq_generic_delay} considered in this paper.

\begin{definition}
\label{def_controllability_EDP}
Let $T > 0$ and $q\in [1,+\infty]$. System~\eqref{eq:hyperbolic} is said to be
\begin{enumerate}[1)]
\item  \emph{$L^q$-approximately controllable in time $T$} if, for every $\epsilon>0$ and $\phi, \psi \in L^q([0,1],\mathbb{R}^n)$, there exists $u \in L^q([0,T],\mathbb{R}^m)$ such that the solution $R$ of \eqref{eq:hyperbolic} with initial condition $\phi$ and control $u$ satisfies $\|R(T,\cdot)-\psi\|_{[0,1],\,q} < \epsilon$.

\item  \emph{$L^q$-exactly controllable in time $T$} if, for every $\phi, \psi \in L^q([0,1],\mathbb{R}^n)$, there exists $u \in L^q([0,T],\mathbb{R}^m)$ such that the solution $R$ of \eqref{eq:hyperbolic} with initial condition $\phi$ and control $u$ satisfies $R(T, \cdot) = \psi$.
\end{enumerate}
\end{definition}

\begin{definition}Let $T > 0$ and $q \in [1,+\infty]$. System~\eqref{eq_generic_delay} is said to be
\begin{enumerate}[1)]
\item \emph{$L^q$-approximately controllable in time $T$} if, 
for every $\epsilon > 0$ and $\phi, \psi \in \Sigma^q$,
there exists $u \in L^q([0,T],\mathbb{R}^m)$ such that the solution $y$ of \eqref{eq_generic_delay} with initial condition $\phi$ and control $u$ satisfies $\|y_{[T]}- \psi\|_{\Sigma^q} < \epsilon$.

\item \emph{$L^q$-exactly controllable in time $T$} if, 
for every 
$\phi, \psi \in \Sigma^q$, there exists $u \in L^q([0,T],\mathbb{R}^m)$ such that the solution $y$ of \eqref{eq_generic_delay} with initial condition $\phi$ and control $u$ satisfies $y_{[T]}= \psi$.

\item \emph{$L^q$-approximately controllable from the origin} if, for every $\epsilon > 0$ and $\psi \in \Sigma^q$, there exist $T_{\epsilon,\psi} > 0$ and $u \in L^q([0,T_{\epsilon,\phi}],\mathbb{R}^m)$ such that the solution $y$ of \eqref{eq_generic_delay} with initial condition $0$ and control $u$ satisfies $\|y_{[T_{\epsilon,\psi}]} - \psi\|_{\Sigma^q} < \epsilon$.

\item \emph{$L^q$-exactly controllable from the origin} if, for every  $\psi \in \Sigma^q$, there exist $T_{\psi}>0$ and $u \in L^q([0,T_{\psi}],\mathbb{R}^m)$ such that the solution $y$ of \eqref{eq_generic_delay} with initial condition $0$ and control $u$ satisfies $y_{[T_{\psi}]} = \psi$.
\end{enumerate}
\end{definition}

Even if we are interested in approximate and exact controllability in a given time $T > 0$, we also 
introduce controllability from the origin in free time 
since this notion 
is the main one that is considered in realization theory. 

Based on Proposition~\ref{prop:equiv-Lq}, the notions of approximate and exact controllability for \eqref{eq:hyperbolic} are equivalent to the corresponding ones for \eqref{eq_generic_delay}, as stated in the following proposition.

\begin{proposition}
\label{prop:control-equiv}
Let $T > 0$ and $q \in [1, +\infty]$. System~\eqref{eq:hyperbolic} is $L^q$-approximate (respectively, exactly) controllable in time $T$ if and only if the same is true for System~\eqref{eq_generic_delay}.
\end{proposition}

Thanks to Proposition~\ref{prop:control-equiv}, we consider from now on controllability issues mainly within the framework of linear difference equations, as defined in System~\eqref{eq_generic_delay}.

By linearity of System~\eqref{eq_generic_delay}, its approximate and exact controllability in time $T$ can be characterized in terms of the operator $E_q(T)$ and its dual 
operator 
$E_q(T)^*$ (that is, reducing to the case where the initial condition is zero), as stated below.

\begin{proposition}
\label{prop:first_charact_control1}
 Let $T > 0$ and $q \in [1,+\infty)$. Then 
\begin{enumerate}
\item The following assertions are equivalent:
\begin{enumerate}[({1}.a)]
\item\label{item:AC} System~\eqref{eq_generic_delay} is $L^q$-approximately controllable in time $T$;
\item\label{item:RangeDense} $\Ran E_q(T)$ is dense in $\Sigma^q$;
\item\label{item:DualInjective} The operator $E_q(T)^*$ is injective.
\end{enumerate}
\item The following assertions are equivalent:
\begin{enumerate}[({2}.a)]
\item\label{item:EC} System~\eqref{eq_generic_delay} is $L^q$-exactly controllable in time $T$;
\item\label{item:RangeAll} $\Ran E_q(T) = \Sigma^q$;
\item\label{item:DualInequality} The operator $E_q(T)^*$ is bounded below, i.e., there exists $c > 0$ such that
\begin{equation}
\label{eq:observability-inequality}
\norm{E_q(T)^* y}_{[0,T],\,q'} \ge c \norm{y}_{\Sigma^{q'}},\qquad y \in \Sigma^{q'}.
\end{equation}
where $q'$ denotes the conjugate exponent of $q$.
\end{enumerate}
\end{enumerate}
\end{proposition}

\begin{proof}
On the one hand, 
the equivalences between \ref{item:AC} and \ref{item:RangeDense}, and between \ref{item:EC} and \ref{item:RangeAll}, follow directly from \eqref{eq:variation_constant}. 
On the other hand, the equivalences between \ref{item:RangeDense} and \ref{item:DualInjective}, and between \ref{item:RangeAll} and \ref{item:DualInequality} follow from classical functional analysis arguments: the first one is a consequence of \cite[Corollary~2.18]{Brezis2011Functional}, while the second one follows from \cite[Theorem~2.20]{Brezis2011Functional}.
\end{proof}

Inequalities such as \eqref{eq:observability-inequality} are usually known as \emph{observability inequalities} due to their link with the observability problem of the corresponding dual system, and they are one of the most common tools used to characterize the controllability of linear PDEs \cite{Tucsnak_Weiss, coron}.
In particular, in \cite{Chitour2020Approximate}, the controllability of the 
linear difference equation~\eqref{system_lin_formel2} in dimension $2$ with $N = 2$ delays is characterized thanks to an inequality of the form \eqref{eq:observability-inequality}. However, an extension of such an approach to systems in higher dimension or with more delays seems difficult due to the 
combinatorics underneath the operator $E_{q}(T)$. This is why we adopt here another approach, based on Hautus-type criteria in the frequency domain, as done in \cite{ChitourHautus} for the 
linear difference equation~\eqref{system_lin_formel2}.

\begin{remark}\label{rem:commensurable}
In the case where the delays $\tau_1,\dots,\tau_n$ are commensurable (i.e., all their pairwise ratios are rational),
it is well-known that
it is possible to reformulate  
\eqref{eq_generic_delay} as an equivalent difference equation with a single delay, at the price of augmenting the dimension of the state space (see, e.g., \cite[Section~3]{Chitour2020Approximate}). Once this is performed, it easily follows that $L^q$-approximate and exact controllability are equivalent (and independent of $q\in [1,+\infty]$) and can be checked by a Kalman criterion. 
\end{remark}

\section{Upper bound on the controllability time}\label{sec:upper-bound-time}

We prove in this section that controllability from the origin 
and controllability in finite time $\tau_1+\dots+\tau_n$ are equivalent. This property is obtained adapting the proof of 
\cite[Theorem~4.7]{ChitourHautus} and relies on Lemma~\ref{lem:1}.

\begin{theorem}
\label{th_appr_minimal_time}
\label{lem:RanE_con}
Set $T_*:=\tau_1+\cdots+\tau_n$. For all $T\geq T_*$ and  $q \in [1,+\infty)$, we have
\begin{equation}
\label{eq:RanE_con1}
\Ran E_{q}(T)=\Ran E_{q}(T_*).
\end{equation}
In particular, 
System~\eqref{eq_generic_delay} is approximately (respectively, exactly) controllable from the origin if and only if it is approximately (respectively, exactly) controllable in time $T$ for every  $T\ge T_*$.
\end{theorem}

\begin{proof}
Note that the last part of the statement follows straightforwardly from \eqref{eq:RanE_con1} and Remark~\ref{remk:RanETnondecre}.

The proof of Equation~\eqref{eq:RanE_con1} is divided in two steps. We first prove that \eqref{eq:RanE_con1} is satisfied for $T \in [T_*, T_* + \tau_{\min}]$, where $\tau_{\min}=\min_{j=1,\dots,m}\tau_j$. We then deduce from the flow formula \eqref{eq:variation_constant} that \eqref{eq:RanE_con1} is actually satisfied for all $T\geq  T_*$.

Let $T > T_*$ and $u \in L^{q}([0,T], \mathbb{R}^m)$.
We define
\begin{equation}
\label{eq:RanE_con2}
u_1(s):=u(s+T-T_*),\qquad \mbox{ $s \in [0,T_*]$,}
\end{equation}
and
\begin{equation}
\label{eq:RanE_con3}
u_2(s):=  -\sum_{\substack{k \in \{0, 1\}^n \setminus \{(0, \dotsc, 0)\} \\ s < \tau \cdot k \le s + T - T_\ast}} \alpha_k u(s - \tau \cdot k + T - T_\ast), \qquad \mbox{$s \in [0,T_*]$},
\end{equation}
where the coefficients  $\alpha_k$ are defined as in Proposition~\ref{lem:1}. 

For $j\in \llbracket 1,n\rrbracket$ and  
$t \in [-\tau_j,0]$, we have
\begin{align}
\left(E_{q}(T)u \right)_j(t)&= \sum_{\substack{\ell \in \mathbb{N}^n \\ \tau \cdot \ell \le {T+t}}}  e_j^T\Xi_\ell B u(T+t-\tau \cdot \ell) \nonumber\\
&=\sum_{\substack{\ell \in \mathbb{N}^n \\ \tau \cdot \ell \le T_*+t}}  e_j^T\Xi_\ell B u(T+t-\tau \cdot \ell)+\sum_{\substack{\ell \in \mathbb{N}^n \\ T_*+t<\tau \cdot \ell \le T+t}} e_j^T \Xi_\ell B u(T+t-\tau \cdot \ell).
\label{eq:RanE_con4}
\end{align}
The first sum of the right-hand side of \eqref{eq:RanE_con4} satisfies
\begin{equation}
\label{eq:RanE_con5}
\sum_{\substack{\ell \in \mathbb{N}^n \\ \tau \cdot \ell \le T_*+t}}  e_j^T\Xi_\ell B u(T+t-\tau \cdot \ell)=(E_{{q}}(T_*)u_1)_j(t).
\end{equation}

Note that, if $\tau \cdot \ell > T_\ast + t$, since $t \in [-\tau_j, 0]$, then $\tau \cdot \ell > T_\ast - \tau_j$, and thus we have $\max_{i \in \llbracket 1, n\rrbracket} \ell_i \geq 2$ or $\ell_j = 1$. We deduce from Lemma~\ref{lem:1} that the second sum of \eqref{eq:RanE_con4} can be written as 
\begin{multline}
\label{eq:RanE_con6}
\sum_{\substack{\ell \in \mathbb{N}^n \\ T_\ast + t < \tau \cdot \ell \le T + t}} e_j^T \Xi_\ell B u(T+t-\tau \cdot \ell)
\\ =-\sum_{\substack{\ell \in \mathbb{N}^n \\ T_\ast + t < \tau \cdot \ell \le T + t}} 
\sum_{k \in \{0, 1\}^n \setminus \{(0, \dotsc, 0)\}} \alpha_k e_j^T \Xi_{\ell - k} B u(T + t -\tau \cdot \ell).
\end{multline}
The substitution $\ell' = \ell - k$ in Equation~\eqref{eq:RanE_con6} yields
\begin{multline}
\label{eq:RanE_con7}
\sum_{\substack{\ell \in \mathbb{N}^n \\ T_\ast + t < \tau \cdot \ell \le T + t}} e_j^T \Xi_\ell B u(T+t-\tau \cdot \ell) \\ = -\sum_{k \in \{0, 1\}^n \setminus \{(0, \dotsc, 0)\}} \sum_{\substack{\ell' \in \mathbb{N}^n \\ T_\ast + t < \tau \cdot (\ell' + k) \le T + t}}  \alpha_k e_j^T \Xi_{\ell'} B u(T + t - \tau \cdot (\ell' + k)).
\end{multline}

Note that, for all $T \in [T_\ast, T_\ast + \tau_{\min}]$, $k \in \{0, 1\}^n \setminus \{(0, \dotsc, 0)\}$, and $\ell'\in \mathbb{N}^n$, if $\tau \cdot (\ell' + k) \le T + t$, then $\tau\cdot \ell'\leq T_*+t$. Hence, for $T \in [T_\ast, T_\ast + \tau_{\min}]$, we can rewrite Equation~\eqref{eq:RanE_con7} as
\begin{align}
\MoveEqLeft[6] \sum_{\substack{\ell \in \mathbb{N}^N \\ T_\ast + t < \tau \cdot \ell \le T + t}}  e_j^T \Xi_\ell B u(T + t - \tau \cdot \ell) \nonumber \\
& = -\sum_{\substack{\ell' \in \mathbb{N}^n \\ \tau \cdot \ell' \le T_\ast + t}} e_j^T \Xi_{\ell'} B 
\sum_{\substack{k \in \{0, 1\}^n \setminus \{(0, \dotsc, 0)\} \\ T_\ast + t < \tau \cdot (\ell'+k) \le T + t}} \alpha_k u(T + t - \tau \cdot (\ell' + k))
\nonumber\\
& = (E_{{q}}(T_\ast)u_2)_j(t).
\label{eq:RanE_con8}
\end{align}
Equations~\eqref{eq:RanE_con4}, \eqref{eq:RanE_con5}, and \eqref{eq:RanE_con8} prove that, for $T \in [T_\ast,T_\ast+\tau_{\min}]$ and 
$t \in [-\tau_j,0]$,
\begin{equation}
\label{eq:RanE_con9}
(E_{q}(T)u)_j(t)=(E_{q}(T_\ast)u_1)_j(t)+(E_{q}(T_\ast)u_2)_j(t)=(E_{q}(T_\ast)(u_1+u_2))_j(t),
\end{equation}
and thus
\begin{equation}
\label{eq:RanE_con10}
\Ran E_{q}(T)=\Ran E_{q}(T_\ast), \qquad T \in [T_\ast, T_\ast+\tau_{\min}].
\end{equation}

Let us now extend Equation~\eqref{eq:RanE_con10} to all  $T \in [T_*,+\infty)$. Let $V=\Ran E_{q}(T_*)$ and $x \in V$. Fix  $u\in L^q([0,T_*],\mathbb{R}^m)$ such that $x=E_{q}(T_*)u$. For $t\in [0,\tau_{\min}]$, define $\tilde{u}\in L^q([0,T_\ast+t],\mathbb{R}^m)$ by setting 
$\tilde{u}|_{[0,T_\ast]}=u$
 and ${\tilde{u}}(s)=0$ for $s \in [T_\ast,T_\ast+t]$. From 
 formula \eqref{eq:variation_constant}, we have
 \begin{equation*}
Upsilon_{q}(t) x= E_{q}(T_\ast+t)\tilde{u} \in \Ran E_{q}(T_\ast+t).
 \end{equation*}
Thanks to   Equation~\eqref{eq:RanE_con10} we have proved that
\begin{equation}\label{eq:invariance}
 \Upsilon_{q}(t)x \in V \qquad \text{ for all } t\in[0,\tau_{\min}],\ x\in V.
 \end{equation}

 Let $y \in \Ran E_{q}(T)$ for $T \in [T_\ast+\tau_{\min},T_\ast+2 \tau_{\min}]$ and $u \in L^q([0,T],\mathbb{R}^m)$ such that $y=E_{q}(T)u$. Define $z=E_{q}(T_\ast+\tau_{\min})u|_{[0,T_\ast+\tau_{\min}]} \in V$. 
 Equation~\eqref{eq:variation_constant} gives
\begin{equation}
\label{eq_lem3:4}
y=\Upsilon_{q}(T-T_\ast-\tau_{\min})z+E_{q}(T-T_\ast-\tau_{\min})\check{u}, 
\end{equation}
where $\check{ u}
(\alpha)=u(\alpha+T_\ast+\tau_{\min})$ for $\alpha \in [0,T-T_\ast-\tau_{\min}]$. We deduce 
from \eqref{eq:RanE_con10} and \eqref{eq:invariance}
that $y \in \Ran E_{q}(T_\ast)$, proving  that 
 \begin{equation}
 \label{eq_lem3:5}
  \Ran E_{q}(T_\ast)=\Ran E_{q}(T), \qquad T \in [T_\ast+\tau_{\min},T_\ast+2 \tau_{\min}].
 \end{equation}
The iteration of the same process proves that Equation~\eqref{eq_lem3:5} actually holds for all $T\ge T_\ast$.
\end{proof}

\section{Controllability criteria}
\label{sec:control_delay_system} 

\subsection{Statement of the controllability criteria}

 We are now in position to state our main results about 
  the approximate and exact controllability of System~\eqref{eq:hyperbolic}.
Denote by   $H\colon \mathbb{C}\to \mathcal{M}_{n, n}(\mathbb{C})$ the function defined by 
\begin{equation}\label{eq:Hp}
H(p)
= \diag(e^{p\tau_1},\dots,e^{p\tau_n})-K,\qquad p\in \mathbb{C}.
\end{equation}

\begin{theorem}
\label{Th_appro_con}
Let $q \in [1,+\infty)$. System~\eqref{eq:hyperbolic} is $L^q$-approximately controllable in time $\tau_1 + \dotsb + \tau_n$ if and only if $\rank [K,B]=n$ and one of the following equivalent assertions holds true:
\begin{enumerate}
\item\label{item:AC-1} $\rank  \left[H(p),B\right]=n$ for every $p\in \mathbb{C}$;
\item\label{item:AC-2} For every $p \in \mathbb{C}$, one has
\begin{equation*}
 \inf \left\{ \vertii{g^T H(p)}+\vertii{g^TB} \mid g\in \mathbb{C}^n,\;\|g^T \|=1 \right\} >0;
\end{equation*}
\item\label{item:AC-3} For every $p \in \mathbb{C}$, one has
\begin{equation*}
\det \left( H(p) H(p)^*+B  B^* \right) >0 .
\end{equation*}
\end{enumerate}

\end{theorem}

\begin{theorem}
\label{Th_exact_con1}
System~\eqref{eq:hyperbolic} is $L^1$-exactly controllable in time $\tau_1 + \dotsb + \tau_n$ if and only if one
of the following
equivalent assertions
holds true:
\begin{enumerate}
\item\label{item:EC-1} $\rank  \left[M,B\right]=n$ for every $M \in \overline{H(\mathbb{C})}$;
\item\label{item:EC-2} There exists $\alpha>0$ such that, for every $p \in \mathbb{C}$,
\begin{equation*}
 \inf \left\{ \vertii{g^TH(p)}+\vertii{g^TB}\mid g\in \mathbb{C}^n,\;\|g^T\|=1 \right\} \ge \alpha;
\end{equation*}
\item\label{item:EC-3} There exists $\alpha>0$ such that, for every $p \in \mathbb{C}$,
\begin{equation*}
\det \left( H(p) H(p)^*+B  B^* \right) \ge \alpha .
\end{equation*}
\end{enumerate}

\end{theorem}

Theorem~\ref{Th_appro_con} gives necessary and sufficient conditions for the $L^q$-approximate controllability of 
linear one-dimensional hyperbolic systems
for $q\in [1,+\infty)$, while Theorem~\ref{Th_exact_con1} states necessary and sufficient conditions for the $L^1$-exact controllability. The necessary condition $\rank [K,B]=n$ stated in Theorem~\ref{Th_appro_con} for the approximate controllability is implied by each item of Theorem~\ref{Th_exact_con1}, by letting the real part of $p$ tend to $-\infty$ in \eqref{eq:Hp}.
We conjecture that the conditions given in Theorem~\ref{Th_exact_con1} are also necessary and sufficient to characterize the $L^q$-exact controllability for $q \in (1,+\infty)$.

The remaining of the paper is devoted to the proofs of Theorems~\ref{Th_appro_con} and \ref{Th_exact_con1}. The equivalence between Items~\ref{item:AC-1}--\ref{item:AC-3} in Theorem~\ref{Th_appro_con} is trivial by simply looking at the negation of each statement. Regarding now the equivalence between Items~\ref{item:EC-1}--\ref{item:EC-3} in Theorem~\ref{Th_exact_con1}, it immediately follows from the following lemma.

\begin{lemma}
\label{lem:equivs-EC}
The following assertions are equivalent:
\begin{enumerate}
\item\label{item:EC-n1} There exists $M \in \overline{H(\mathbb{C})}$ such that $\rank  \left[M,B\right] < n$;
\item\label{item:EC-n2} There exist sequences $(p_k)_{k \in \mathbb N} \in \mathbb C^{\mathbb N}$ and $(g_k)_{k \in \mathbb N} \in (\mathbb C^n)^{\mathbb N}$ such that $\norm{g_k} = 1$ for every $k \in \mathbb N$ and
\[
g_k^T H(p_k) \to 0 \text{ and } g_k^T B \to 0 \qquad \text{ as } k\to +\infty;
\]
\item\label{item:EC-n3} There exists a sequence $(p_k)_{k \in \mathbb N} \in \mathbb C^{\mathbb N}$ such that
\begin{equation*}
\det \left( H(p_k) H(p_k)^*+B  B^* \right) \to 0 \qquad \text{ as } k \to +\infty.
\end{equation*}
\end{enumerate}
\end{lemma}

\begin{proof}
Assume that Item~\ref{item:EC-n1} holds true, i.e., there exist
$M\in \overline{H(\mathbb{C})}$ and $g\in \mathbb C^n$ of norm one so that $g^T M = 0$, and $g^T B = 0$. Let $(p_k)_{k \in \mathbb N}$ be such that $H(p_k) \to M$ as $k \to +\infty$. Item~\ref{item:EC-n2} follows  by taking $g_k\equiv g$, while Item~\ref{item:EC-n3} follows  from the fact that $\det(M M^\ast + B B^\ast) = 0$ and the continuity of the determinant function.

Assuming now either Item~\ref{item:EC-n2} or Item~\ref{item:EC-n3}, let us deduce Item~\ref{item:EC-n1}.
We first show that the real part of the sequences $(p_k)_{k \in \mathbb N}$ from Items~\ref{item:EC-n2} and \ref{item:EC-n3} must have real part  bounded from above. Reasoning by contradiction we have that, up to extracting a subsequence, $\Re(p_k)$ tends to $+\infty$ as $k$ tends to $+\infty$. Noticing that $H(p_k) = \diag(e^{p_k \tau_1}, \dotsc, e^{p_k \tau_n}) (I_n - \diag(e^{-p_k \tau_1}, \dotsc, e^{-p_k \tau_n}) K)$, it follows that $\norm{g_k^T H(p_k)} \to +\infty$ 
and $\det\left( H(p_k) H(p_k)^*+B  B^* \right)=(1+o(1))\prod_{j=1}^n e^{2 \Re(p_k)\tau_j} \to +\infty$
as $k \to +\infty$. 
Hence, we deduce that there exists $M \in \overline{H(\mathbb{C})}$ such that, up to extracting a subsequence,  $H(p_k) \to M$ as $k \to +\infty$, and in this case Item~\ref{item:EC-n2} would imply, up to extracting a converging subsequence of $(g_k)_{k \in \mathbb N}$, the existence of $g \in \mathbb C^n$ with $\norm{g} = 1$ and $g^T [M, B] = 0$, while Item~\ref{item:EC-n3} would imply that $\det(M M^* + B B^*) = 0$, both yielding Item~\ref{item:EC-n1}.
\end{proof}

It remains to prove that $L^q$-approximate (respectively, $L^1$-exact) controllability is equivalent to $\rank [K, B] = n$ and one of Items~\ref{item:AC-1}--\ref{item:AC-3} of Theorem~\ref{Th_appro_con} (respectively, Items~\ref{item:AC-1}--\ref{item:AC-3} of Theorem~\ref{Th_exact_con1}). To get such a result, our approach is based on realization theory.

\subsection{Realization theory approach}

We now explain how to reformulate controllability problems for \eqref{eq_generic_delay} within the framework of the realization theory developed by Y.~Yamamoto in \cite{yamamoto1989reachability, Yutaka_Yamamoto}, cf.\ also \cite[Section 5]{ChitourHautus}. In such a framework, a control system is defined through an input--output relation (with zero initial condition), where the input is a function whose support is compact and included in $\mathbb R_-$ and the output is observed for all nonnegative times. 
In order to apply such a framework to our controllability problem, we should fulfill the following conditions:
\begin{enumerate}
 \item  the input $u$ is taken in the space  $\Omega^q$ of all functions of $L^q (\mathbb{R},\mathbb{R}^k)$ whose support is compact and contained in the interval $(-\infty,0]$;
  \item The initial state is equal to zero, i.e., for 
 $t$ smaller than the infimum of the support of $u$,
 we have $y_j(t)=0$ for all $j \in \llbracket 1,n\rrbracket$, and  $y$ satisfies
\begin{equation}
\label{eq:input-state}
\begin{pmatrix}
y_1(t) \\
\vdots\\
y_n(t)
\end{pmatrix}=K \begin{pmatrix}
y_1(t-\tau_1) \\
\vdots\\
y_n(t-\tau_n)
\end{pmatrix}+Bu(t), \qquad t \in \mathbb{R};
\end{equation}

 \item \label{item:Q} The output $z\colon\mathbb{R}_+\to\mathbb{R}^n$ is computed from the trajectory $y$ of System~\eqref{eq:input-state}
 and must belong to 
 \begin{equation}
 \label{def_XQ}
 X^{Q,q}:=\left\{z \in L^q_{\mathrm{loc}}\left(\mathbb{R}_+,\mathbb{R}^n\right)|\, \pi(Q*z)=0  \right\},
 \end{equation}
 where $\pi$ is the operator of truncation on  $\mathbb{R}_+$ and $Q$ is a convolution kernel defined so that the relation $\pi(Q*z)=0$ expresses the dynamics of System~\eqref{eq:input-state} for $t\geq 0$.
\end{enumerate} 
We next aim at characterizing the output $z$ and the convolution kernel $Q$ for trajectories $y$ of System~\eqref{eq:input-state} 
associated with controls $u\in \Omega^q$. Note that our controllability issues amount to control $y_{[T]}\in \Sigma^q$, $T \geq 0$, and for $t\geq 0$, $y$ is a trajectory of System~\eqref{eq_generic_delay} if and only if $\pi(\tilde{Q}*y)=0$, where $\tilde{Q}$ is given by
\[
\tilde{Q}=I_n\delta_0-\sum_{j=1}^n Ke_je_j^T\delta_{\tau_j}.
\]
To fit Item~\ref{item:Q}, the output $z$ should represent $y_{[T]}$, $T\ge 0$, while being defined for $t\geq 0$. We then define it as 
\begin{equation}\label{eq:output-z}
z(t)=
\begin{pmatrix}
 z_1(t) \\
 \vdots\\
 z_n(t)
 \end{pmatrix}= \begin{pmatrix}
 y_1(t-\tau_1) \\
 \vdots\\
 y_n(t-\tau_n)
 \end{pmatrix},\qquad t \in [0,+\infty),
\end{equation}
and the relation $\pi(\tilde{Q}*y)=0$  becomes $\pi(Q*z)=0$, where the convolution kernel $Q$ is  given by
\begin{equation}
\label{def:Q}
Q=\diag(\delta_{-\tau_1},\dots,\delta_{-\tau_n})-K\delta_0.
\end{equation}
With $Q$ defined as above, the set $X^{Q,q}$ from \eqref{def_XQ} can be easily identified with the space $\Sigma^q$. Indeed, one first sees that 
$X^{Q,q}$  is identified with $\prod_{j=1}^n L^q ([0, \tau_j], \mathbb R)$ since $z \in X^{Q,q}$ if and only if, for each $j=1,\dots,n$, the restriction $z_j|_{[0,\tau_j]}$ is in $L^q([0,\tau_j],\mathbb{R})$  and $z$ is the unique extension of $(z_1|_{[0,\tau_1]},\dots,z_n|_{[0,\tau_n]})^T$ on the half-line $[0,+\infty)$ satisfying the condition $\pi(Q*z)=0$. In a second step, one identifies $\prod_{j=1}^n L^q ([0, \tau_j], \mathbb R)$ with $\Sigma^q$ by a translation of each component.

With the above definitions, the input-output system described by \eqref{eq:input-state} and \eqref{eq:output-z} can be written as
\[
Q * z = P * u\qquad\mbox{on}\quad \mathbb{R},
\]
where $P = B \delta_0$ and $z$ is extended on $\mathbb{R}_-$ 
by computing it from $y$ according to \eqref{eq:output-z}. 

We can reformulate the controllability of System~\eqref{eq_generic_delay} 
in terms of the input-output delay system \eqref{eq:input-state}--\eqref{eq:output-z} as follows.

\begin{lemma} Let $q \in [1,+\infty)$. System~\eqref{eq_generic_delay} is
\begin{enumerate}[1)]
\item {$L^q$-approximately controllable from the origin} if and only if for each $\phi \in X^{Q,q}$ there exists a sequence of inputs $(u_n)_{n \in \mathbb{N}}$  in $\Omega^q$ 
whose associated sequence of outputs 
$(z_n)_{n\in\mathbb{N}}$
satisfies
\[
z_n \underset{n \to +\infty}{\longrightarrow} \phi \qquad \text{in } L^q_{\rm loc}\left(\mathbb{R}_+,\mathbb{R}^n\right);
\]
\item {$L^q$-exactly controllable from the origin} if and only if for each $\phi \in X^{Q,q}$ there exists $u \in \Omega^q$ such that $z = \phi$.
\end{enumerate}
\end{lemma}

Note that, in the above formulation of the controllability problem for System~\eqref{eq_generic_delay} in terms of 
realization theory, the main differences with respect to \cite{ChitourHautus} lie in the state space (here equal to $\Sigma^q = \prod_{i=1}^n L^q([-\tau_i,0],\mathbb{R})$ instead of $L^q([-\tau_{\max},0],\mathbb{R}^n)$ in \cite{ChitourHautus}), and the definition of the distribution $Q$ in \eqref{def:Q} (which was equal to $\delta_{-\Lambda_N} I_d- \sum_{j=1}^N \delta_{-\Lambda_N+\Lambda_j} A_j$ in \cite{ChitourHautus}). Moreover, notice that here $H = \widehat{Q}$, 
leading to
some simplification with respect to \cite{ChitourHautus}.

We are now in position to complete the proof of Theorems~\ref{Th_appro_con} and \ref{Th_exact_con1}.

\begin{proof}[Proof of Theorem~\ref{Th_appro_con}]
Recall that, thanks to Theorem~\ref{th_appr_minimal_time}, $L^q$-approximate controllability in time $\tau_1 + \dotsb + \tau_n$ is equivalent to $L^q$-approximate controllability from the origin. The latter, in turn, is equivalent to $L^2$-approximate controllability from the origin, as follows by continuity of the input-output maps $\Omega^q\ni u\mapsto z\in X^{Q, q}$, $q\in [1,+\infty)$ (see the proof of \cite[Theorem~5.7]{ChitourHautus} for details). To conclude, it suffices to apply the characterization of $L^2$-approximate controllability from the origin from \cite[Corollary~4.10]{yamamoto1989reachability} (where the latter property is referred to as \emph{quasi-reachability}).
\end{proof}

\begin{proof}[Proof of Theorem~\ref{Th_exact_con1}]
We first claim that $L^1$-exact controllability of \eqref{eq:hyperbolic} in time $\tau_1 + \dotsb + \tau_n$ is equivalent to the existence of two matrix-valued Radon measures $R,S$ with entries in $M(\mathbb{R}_-)$ such that
\begin{equation}\label{eq:bezout}
Q\ast R+P\ast S=\delta_0 I_n.
\end{equation}
Equation~\eqref{eq:bezout} is referred to as a \emph{B\'ezout identity}. Indeed, the necessity of \eqref{eq:bezout} follows from the same ideas as those used in the proof of \cite[Theorem~5.13]{ChitourHautus} but, due to some nontrivial notational adaptations to our setting, we provide a complete argument in the Appendix. The converse implication is obtained by proceeding exactly as in the arguments of \cite[Proposition~5.10 and Corollary~5.12]{ChitourHautus}.

If \eqref{eq:bezout} is satisfied for some matrix-valued Radon measures $R,S$ with entries in $M(\mathbb{R}_-)$, then, by proceeding as in the proof of \cite[Proposition~5.17]{ChitourHautus}, we deduce that Item~\ref{item:EC-2} holds true. The latter is equivalent to
Items~\ref{item:EC-1} and \ref{item:EC-3} by Lemma~\ref{lem:equivs-EC}.

To prove the converse, let us first introduce, for $T > 0$, the subspace $\Omega_{-}^{T}$ of $M(\mathbb{R}_-)$ made of the elements $h\in M(\mathbb{R}_-)$ of the form
\begin{equation}
\label{eq:omega}
h = \sum_{j=0}^{N} h_j \delta_{-\lambda_j} \quad \text{for some } N \in \mathbb N \text{ and } \lambda_j\in [0,T],\,  h_j \in \mathbb{R} \text{ for } j \in \llbracket 0, N\rrbracket.
\end{equation}
We want to prove that Item~\ref{item:EC-2} of the statement implies the existence of matrix-valued Radon measures $R,S$ with entries in $M(\mathbb{R}_-)$ such that \eqref{eq:bezout} holds true. By proceeding as in \cite{Fuhrmann_corona} (see also the discussion before Conjecture~5.18 in \cite{ChitourHautus}), it is enough to prove that the following \emph{corona problem} admits a solution: given $k \in \mathbb N^\ast$, $T > 0$, $\alpha > 0$, and $f_1, \dotsc, f_k$ in $\Omega_-^T$ satisfying
\[
\sum_{i=1}^k \abs*{\widehat f_i(p)} \geq \alpha \qquad \text{for all } p \in \mathbb C,
\]
find $g_1, \dotsc, g_k$ in $M(\mathbb R_-)$ such that
\[
\sum_{i=1}^k f_i * g_i = \delta_0.
\]
It is proved in \cite[Theorem~4.2]{fueyo-chitour} that this corona problem admits a solution, yielding the conclusion.
\end{proof}

\section{Application to flows in networks}
\label{sec:appli} 

In this section, we apply the controllability results from Section~\ref{sec:control_delay_system} to systems of transport equations describing flows in networks (see, e.g., \cite{KramarSikolya2005, Gantouh2020ApproximateCO}).

\subsection{Statement of the problem}

Consider a 
directed graph $\mathcal{G} = (\mathcal{V}, \mathcal{E})$, where $\mathcal V = \{v_1,\dotsc,v_k\}$ and $\mathcal E = \{\varepsilon_1,\dotsc,\varepsilon_n\}
$ denote, respectively, the sets of vertices and edges of $\mathcal G$. Given an edge $\varepsilon_j \in \mathcal E$, we denote its endpoints by $\varepsilon_j(0)$ and $\varepsilon_j(1)$, and we assume that the edge is oriented from $\varepsilon_j(1)$ to $\varepsilon_j(0)$. We say that $\varepsilon_j$ is an \emph{outgoing} edge of the vertex $\varepsilon_j(1)$ and an \emph{incoming} edge of the vertex $\varepsilon_j(0)$.

The 
\emph{outgoing} 
and \emph{incoming incidence matrices} of the graph $\mathcal G$
are the $k\times m$ matrices $\mathcal{I}^-
$ and $\mathcal{I}^+
$, 
whose coefficients are, respectively,
\begin{equation*}
i^{-}_{ij}=
 \begin{cases}
1, & \mbox{ if $v_i=\varepsilon_j(1)$},\\
0, & \mbox{ otherwise},
\end{cases}\qquad i^{+}_{ij}=
 \begin{cases}
1, & \mbox{ if $v_i=\varepsilon_j(0)$},\\
0, & \mbox{ otherwise},
\end{cases}\qquad i \in \llbracket 1, k\rrbracket,\ j \in \llbracket 1, n\rrbracket.
\end{equation*} 
Note that each column of $\mathcal I^-$ and of $\mathcal I^+$ has exactly one nonzero element.

With each $(i, j) \in \llbracket 1, k\rrbracket \times \llbracket 1, n\rrbracket$ we associate a weight $w_{i, j}^- \in [0, 1]$, with the assumption that \begin{equation}
    \label{eq:kirchoff}
\sum_{j=1}^{n} w_{i j}^- = 1,\qquad i \in \llbracket 1, k\rrbracket,
\end{equation} 
and that $w_{ij}^-=0$ if and only if $i_{ij}^-=0$. We define the \emph{weighted outgoing incidence matrix} $\mathcal{I}^-_{w}=(w_{ij}^- )_{i\in\llbracket 1,k\rrbracket,j\in \llbracket1,n\rrbracket}$.

We now define a controlled flow on the graph $\mathcal G$ following the approach of \cite{KramarSikolya2005, Gantouh2020ApproximateCO}. Each edge $\varepsilon_j$ is identified with a real interval which, up to a normalization, can be assumed to be the interval $[0, 1]$, with the values $0$ and $1$ corresponding to the endpoints $\varepsilon_j(0)$ and $\varepsilon_j(1)$ of $\varepsilon_j$, respectively. On each such edge, there is a flow 
from the endpoint $1$ to the endpoint $0$ of the corresponding interval, and such a flow is described by a transport equation. At each vertex $v_i \in \mathcal V$ of the graph, we assume that the total incoming flow is distributed to the outgoing edges $\varepsilon_j$ according to the weights $w_{ij}^-$. In addition, we also have $m$ scalar controls $u_1, \dotsc, u_m$ acting on the vertices of the graph. The action of the control $u_l$ at the vertex $v_i$ is 
leveraged by a real coefficient $\gamma_{il}$, and the total control acting on a vertex $i$ is distributed to the outgoing edges $\varepsilon_j$ according to the weights $w_{ij}^-$.
Condition~\eqref{eq:kirchoff}  can then be interpreted as the conservation of the mass at the vertices.  This interpretation motivates  the following assumption on the graph $\mathcal G$, which is used in \cite{KramarSikolya2005} and which is implied by the strong connectivity assumption of \cite{Gantouh2020ApproximateCO}. 
\begin{enumerate}[label={(\Alph*)}, start=11]
\item\label{hypo:graph}
For every $v_i \in \mathcal V$, there exist $\varepsilon_{j}, \varepsilon_{j'} \in \mathcal E$ such that $\varepsilon_j(0) = v_i$ and $\varepsilon_{j'}(1) = v_i$.
\end{enumerate}
In other words, Assumption~\ref{hypo:graph} states that every vertex has at least one outgoing  and one incoming edge. In particular,  $k \le n$.
We shall consider that Assumption~\ref{hypo:graph} holds true in the sequel of the section.

The above description of the flow in the graph $\mathcal G$ is represented mathematically by the system
 \begin{equation}\label{eq_gantouh}
\begin{dcases} 
\partial_t z_j(t,x) +
\lambda_j(x) \partial_x z_j(t,x)+d_j(x) z_j(t,x) = 0,\qquad\quad & t >0,\ x \in (0, 1),\\
i^-_{ij} z_{j}(t,1)=w_{ij}^-\sum_{e=1}^n i^+_{ie} z_e(t,0)+ w_{ij}^-\sum_{l=1}^{m} \gamma_{il} u_l(t),&i\in \llbracket1,k\rrbracket,\ j\in \llbracket1,n\rrbracket,
\end{dcases}
\end{equation}
where $d_1,\dotsc,d_n$ belong to $L^1([0, 1], \mathbb R)$ and $\lambda_1, \dotsc, \lambda_n, \frac{1}{\lambda_1}, \dotsc, \frac{1}{\lambda_n}$ belong to $L^\infty([0, 1], \mathbb R_{-}
)$. 
Even though the second line of \eqref{eq_gantouh} formulates $kn$ equations, it provides only $n$ nontrivial conditions, namely the ones corresponding to pairs $(i, j) \in \llbracket 1, k\rrbracket \times \llbracket 1, n\rrbracket$ with $\varepsilon_j(1) = v_i$.

Let us now write System~\eqref{eq_gantouh} under the form \eqref{eq:hyperbolic}. For $x \in [0,1]$, let 
 \[\Lambda(x)= \diag(\lambda_1(x),\dotsc,\lambda_n(x)),\quad D(x)= \diag(d_1(x),\dots,d_n(x)),\quad \Gamma=(\gamma_{il})_{i\in \llbracket1,k\rrbracket,\;l\in\llbracket1,m\rrbracket}.\]
Setting $z(t,x)=(z_1(t,x),\dots,z_n(t,x))^T$, we can rewrite System~\eqref{eq_gantouh} as
\begin{equation}\label{eq_hypgantouh1}
\begin{cases}
\partial_t z(t,x) + \Lambda(x) \partial_x z(t,x) + D(x) z(t,x) = 0, \\
z(t,1) = \left( \mathcal{I}_w^- \right)^T \mathcal{I}^+ z(t,0) + \left(\mathcal{I}_w^-\right)^T\Gamma u(t),
 \end{cases}
 \qquad t >0,\ x \in (0, 1).
\end{equation}
According to the results from Section~\ref{sec:transformation}, System~\eqref{eq_hypgantouh1} can be transformed into the linear difference equation \eqref{eq_generic_delay} with
\begin{equation}
    \label{eq:KandB}
    K=\left( \mathcal{I}_w^- \right)^T \mathcal{I}^+ Z, \qquad B = \left(\mathcal{I}_w^-\right)^T\Gamma, \qquad \tau_j = \int_0^1 \frac{dx}{|\lambda_j(x)|},\; j \in \llbracket 1, n\rrbracket,
\end{equation}
where
\[Z=\diag(e^{-\zeta_1},\dots,e^{-\zeta_n})\qquad \mbox{and}\qquad \zeta_j= \int_0^1 \frac{d_j(x)}{|\lambda_j(x)|} dx\quad  \mbox{for }j\in \llbracket 1,n\rrbracket.\]

\begin{remark}
In \cite{Gantouh2020ApproximateCO}, the authors considered \eqref{eq_gantouh} in the case $\tau_1 = \dotsb = \tau_n$ and $D\equiv 0$, stating a Kalman-type criterion for the $L^2$-approximate controllability. Thanks to Remark~\ref{rem:commensurable}, in this case, we obtain at once that $L^q$-approximate and exact controllability are equivalent (and independent of $q\in [1,\infty]
$) and that they can be characterized by the Kalman criterion on the matrices $K$ and $B$ from \eqref{eq:KandB}. In addition, we will also show (see Proposition~\ref{prop:circle} below) that a necessary condition for controllability is that $\mathcal G$ is a finite union of directed cycle graphs, in which case $K$ and $B$ are given by \eqref{eq:Kinblocs} below.
\end{remark}

\subsection{A topological necessary condition for controllability}

We next provide our first controllability result, which introduces a strong restriction on the topology of the graph for  \eqref{eq_gantouh} to be controllable.
\begin{proposition}\label{prop:circle}
Assume that System~\eqref{eq_gantouh} is $L^q$-approximate controllable for some $q\in[1,\infty)$. Then the directed graph $\mathcal{G}$ must be the finite union of 
directed cycle graphs.
\end{proposition} 
To prove the above proposition, we need the two following technical results.
\begin{lemma}\label{lem:lin-comb}
Let the matrices $K$ and $B$ be as in \eqref{eq:KandB}.
Then the columns of  $B$ are linear combinations of the columns of  $K$.
In particular, the range of the matrix $[K, B]$ is equal to that of $K$.
\end{lemma}
\begin{proof}
Define first a map $J\colon \llbracket1,k\rrbracket\to \llbracket1,n\rrbracket$ 
 associating  with each $i \in \llbracket 1, k\rrbracket$ a label $j = J(i) \in \llbracket 1, n\rrbracket$ of an edge $\varepsilon_j$ incoming at the vertex $v_i$, i.e., such that $\varepsilon_j(0) = v_i$.
Such a map is well-defined by Assumption~\ref{hypo:graph}. The key point is to notice that, 
for $i\in \llbracket1,k\rrbracket$ and $j\in \llbracket1,n\rrbracket$,
\begin{equation}
\label{eq:KjJi}
K_{jJ(i)}=\sum_{l=1}^kw^-_{lj}i^+_{lJ(i)}e^{-\zeta_{J(i)}}=w^-_{ij}e^{-\zeta_{J(i)}}.
\end{equation}
Moreover, by definition of the matrix $B$, for $j\in \llbracket1,n\rrbracket$ and $r\in \llbracket1,m\rrbracket$,
\[
B_{jr}=\sum_{i=1}^k w^-_{ij}\gamma_{ir},
\]
implying that 
\begin{equation}\label{eq:BK}
B_r=\sum_{i=1}^k\gamma_{ir}e^{\zeta_{J(i)}}K_{J(i)},
\end{equation}
where $B_r$ and $K_{J(i)}$ denote the $r$th column of $B$ and the $J(i)$th column of $K$, respectively.
\end{proof}

\begin{lemma}
\label{lem:K-prop-cols}
If a vertex of $\mathcal{G}$ admits 
at least two incoming edges, then the corresponding matrix $K$ introduced in \eqref{eq:KandB} admits two proportional columns.  
\end{lemma}

\begin{proof}
Let $v_i\in \mathcal{V}$ and $\varepsilon_{j_1},\varepsilon_{j_2}\in\mathcal{E}$ be such that 
$\varepsilon_{j_1}(0) = \varepsilon_{j_2}(0) = v_i$. 
The same computation as in \eqref{eq:KjJi} shows that the $j_1$th and $j_2$th column of $K$ are given respectively by
$K_{j_1} = e^{-\zeta_{j_1}}(w^-_{ij})_{j\in \llbracket 1, n\rrbracket}$ and $K_{j_2} = e^{-\zeta_{j_2}} (w^-_{ij})_{j \in \llbracket 1, n\rrbracket}$ and hence 
$K_{j_1} = e^{\zeta_{j_2}-\zeta_{j_1}}K_{j_2}$.
\end{proof}

We can now prove Proposition~\ref{prop:circle}.

\begin{proof}[Proof of Proposition~\ref{prop:circle}]
By Theorem~\ref{Th_appro_con}, a necessary condition for the $L^q$-approximate controllability of \eqref{eq_gantouh}, for some $q\in [1,\infty)$, is that $\rank [K, B] = n$ and the latter implies, according to Lemmas~\ref{lem:lin-comb} and \ref{lem:K-prop-cols}, that every vertex has exactly one incoming edge.

Let $\Omega\colon \mathcal E \to \mathcal E$ be the map associating with an edge $\varepsilon_{j_1} \in \mathcal E$ the unique edge $\varepsilon_{j_2} \in \mathcal E$ such that $\varepsilon_{j_2}(0) = \varepsilon_{j_1}(1)$. We claim that $\Omega$ is surjective. Indeed, given an edge $\varepsilon_j \in \mathcal E$, let $v_i = \varepsilon_j(0)$. By Assumption~\ref{hypo:graph}, there exists $\varepsilon_{j'} \in \mathcal E$ such that $\varepsilon_{j'}(1) = v_i$, and the uniqueness of the incoming edge at $v_i$ proves that $\Omega(j') = j$.

Since $\Omega$ is a surjective map from the finite set $\mathcal E$ into itself, then $\Omega$ is a bijection, and hence a permutation of $\mathcal E$. Henceforth, $\Omega$ can be decomposed in the product of disjoint cycles, each of them corresponds to a directed cycle graph, and the disjoint union of these directed cycle graphs is equal to $\mathcal G$.
\end{proof}

\subsection{Hautus--Yamamoto tests for controllability of flows in networks}
\label{sec:HY-networks}

In this section, we investigate the $L^q$-approximate and exact controllability of System~\eqref{eq_gantouh}. By taking into account Proposition~\ref{prop:circle}, this amounts to  applying Theorems~\ref{Th_appro_con} and \ref{Th_exact_con1}
to the case where $\mathcal{G}$ is the disjoint union of directed cycle graphs.

To proceed with the second step, we need the following notation.
Given $h \in \mathbb N^\ast$, we denote in this section by $C_h$ the cyclic permutation matrix of size $h \times h$, defined by
\[
C_h = \begin{pmatrix}
0 & 0 & \cdots & 0 & 1 \\
1 & 0 & \cdots & 0 & 0 \\
0 & \ddots & \ddots & \vdots & \vdots \\
\vdots & \ddots & \ddots & 0 & 0 \\
0 & \cdots & 0 & 1 & 0
\end{pmatrix}
\]
(with the convention $C_1 = \begin{pmatrix}1\end{pmatrix}$).
Note that $C_h e_h = e_1$ and $C_he_j=e_{j+1}$ for $j \in \llbracket 1, h-1\rrbracket$, where $(e_1,\dotsc,e_h)$ denotes the canonical basis of $\mathbb{R}^h$. 

When dealing 
with a graph $\mathcal G$ that is the disjoint union of directed cycle graphs, we will denote by $L$ the number of directed cycle graphs in $\mathcal G$. For $l \in \llbracket 1, L\rrbracket$, let $h_l$ denote the number of edges of the $l$th cycle of $\mathcal G$, and set
\[\mathcal J_l = \left\llbracket 1 + \sum_{r=1}^{l - 1} h_r, \sum_{r=1}^l h_r \right\rrbracket.\]

We relabel the edges of $\mathcal G$ in such a way that $\varepsilon_j$ belongs to the $l$th cycle of $\mathcal G$ if and only if $j \in \mathcal J_l$, and we assume in addition that, for each $l \in \llbracket 1, L\rrbracket$, we have $\varepsilon_j(0) = \varepsilon_{j+1}(1)$ for every $j \in \mathcal J_l \setminus \{\max \mathcal J_l\}$ and $\varepsilon_{\max \mathcal J_l}(0) = \varepsilon_{\min \mathcal J_l}(1)$, i.e., if one follows the edges of a cycle, their indices increase by $1$ when one goes from one edge to the next one, until the last edge of the cycle, which leads back to the first one.

Any graph $\mathcal G$ made of a disjoint union of directed cycles has necessarily as many vertices as edges, i.e., $k = n$. In this case, up to relabeling the vertices of $\mathcal G$, we assume that $\varepsilon_j(1) = v_j$ for every $j \in \llbracket 1, n\rrbracket$, and this relabeling yields $\mathcal I^- = I_n$ and $\mathcal I^+ = \diag(C_{h_l})_{l=1}^L$. In addition, Equation~\eqref{eq:kirchoff} implies that $\mathcal I_w^- = \mathcal I^-$, and hence we have from \eqref{eq:KandB} that
\begin{equation}\label{eq:Kinblocs}
K = \diag(C_{h_l}Z_l)_{l=1}^L, \qquad B = \Gamma, 
\end{equation}
where $Z_l=\diag(e^{-\zeta_j})_{j\in \mathcal J_j}$.
 
Let us provide the following result, which deals with the first condition for controllability in Theorems~\ref{Th_appro_con} and \ref{Th_exact_con1}, namely the rank condition on $[K, B]$.

\begin{lemma}
\label{lem:rank-iff-cycles}
With the notations above, 
$\rank[K, B] = n$ if and only if $\mathcal{G}$ is the disjoint union of directed cycle graphs. 
\end{lemma}

\begin{proof}
Taking into account Lemma~\ref{lem:lin-comb} and the argument of Proposition~\ref{prop:circle}, it remains to prove that $\rank K = n$ if $\mathcal{G}$ is the disjoint union of directed cycle graphs. The conclusion follows using \eqref{eq:Kinblocs} and the fact that each matrix $C_{h_l}$ is invertible.
\end{proof}

We next provide a technical result which will be used to refine the spectral condition given in Item~\ref{item:AC-1} from Theorem~\ref{Th_appro_con} in the case where $\mathcal G$ is a disjoint union of directed cycle graphs.
For that purpose, we recall that the matrix $H(p)$ defined in \eqref{eq:Hp}
is now given by 
\[
H_{\rm net}(p)=\diag(\mathcal{T}_l(p)-C_{h_l}Z_l)_{l=1}^L,
\]
where $\mathcal{T}_l(p) = \diag(e^{p\tau_j})_{j \in \mathcal J_l}$.

\begin{lemma}\label{lem:Dl}
Let $\mathcal{D}_l$ for $l\in\llbracket 1,L\rrbracket$ and $\mathcal{R}_{\mathcal{G}}$ be the subsets of the complex plane $\mathbb{C}$ defined by 
\[
\mathcal{D}_l=\left\{-\frac{\sum_{j\in \mathcal{J}_l}\zeta_j}{\sum_{j\in \mathcal{J}_l}\tau_j}+i\frac{2k\pi}{\sum_{j\in \mathcal{J}_l}\tau_j}\suchthat k\in\mathbb{Z}\right\},
\qquad 
\mathcal{R}_{\mathcal{G}}=
\bigcup_{l=1}^L\mathcal{D}_l.
\]
Then the following holds true.
\begin{enumerate}
    \item \label{item:uno}
    The matrix $\mathcal{T}_l(p)-C_{h_l}Z_l$ is invertible if and only if $p\notin \mathcal{D}_l$. As a consequence, the matrix $H_{\rm net}(p)$  is invertible if and only if $p\notin \mathcal{R}_{\mathcal{G}}$.
    \item\label{item:due} For $p\in \mathcal{D}_l$, the rank of $\mathcal{T}_l(p)-C_{h_l}Z_l$ is equal to $h_l-1$ and its range is equal to the orthogonal space 
    to  the vector $y_l(p)=(y_{lj}(p))_{j=1}^{h_l}$ whose components are given by
    \begin{equation}\label{eq:zl}
    y_{lj}(p) = \prod_{t = \min\mathcal J_l}^{j - 2 + \min\mathcal J_l} e^{p \tau_t + \zeta_t} ,\qquad j\in\llbracket 1,h_l\rrbracket.
    \end{equation}
\end{enumerate}
\end{lemma}

\begin{proof}
Fix $l\in\llbracket 1,L\rrbracket$. 
The subset of  
$\mathbb{C}$ for which the matrix $\mathcal{T}_l(p)-C_{h_l}Z_l$ is 
singular consists of 
the complex numbers $p$ so that there exists a nonzero vector $x\in\mathbb{C}^{h_l}$ such that $\mathcal{T}_l(p)x = C_{h_l} Z_l x$. Easy computations yield that 
\[
x_j=\left(\prod_{t=1+\min\mathcal{J}_l }^{j-1+\min\mathcal{J}_l} e^{-(\zeta_{t-1} + p\tau_t)} \right) x_1, \quad j\in \llbracket 2, h_l\rrbracket,
\qquad x_1=e^{-\left(\zeta_{\max\mathcal{J}_l}+p\tau_{\min\mathcal{J}_l}\right)}x_{h_l}.
\]
One deduces that $\mathcal{T}_l(p)-C_{h_l}Z_l$ is 
singular if and only if $p\in\mathbb{C}$  solves the equation
\[
\prod_{t \in \mathcal J_l} e^{\zeta_t + p\tau_t} = 1,
\]
i.e., $p\in \mathcal{D}_l$. This concludes the proof of Item~\ref{item:uno}. 

Moreover, if $p\in \mathcal{D}_l$, it is clear from the previous computation that the kernel of $\mathcal{T}_l(p)-C_{h_l}Z_l$ has dimension one. In addition, its range is the orthogonal to the kernel of $\mathcal{T}_l(p)-Z_lC_{h_l}^T$. Identical computations as above yield that this kernel is 
equal to $\mathbb{C}y_l(p)$ with $y_l(p)$ given in \eqref{eq:zl}. The proof of Item~\ref{item:due} and that of the lemma are complete.
\end{proof}

We are now able to state our main
approximate controllability result for System~\eqref{eq_gantouh}. We first need the following notation.

For $p\in \mathcal D_l$ let $\tilde y_l(p)\in \mathbb C^n$ be the vector obtained from the vector $y_l(p)$ introduced in Item~\ref{item:due} of Lemma~\ref{lem:Dl}
by identifying $\mathbb C^{h_l}$  with the subspace of $\mathbb C^n$ corresponding to the indices in $\mathcal J_l$. 

For every $p\in \mathcal{R}_{\mathcal{G}}$, let
\[ V(p)=\Span\{\tilde y_l(p)\mid l\in \llbracket1,L\rrbracket \text{ such that } p\in \mathcal D_l\}.\]
Notice that $\mathcal{R}_{\mathcal{G}}\ni p\mapsto V(p)$ takes finitely many values if, for every $l\in \llbracket 1,L\rrbracket$, the scalars $\tau_j$, $j\in \mathcal{J}_l$, are commensurable, i.e., rational multiples of a common real number.

\begin{theorem}
\label{th_appro_network_gantouh}
For $q \in [1,+\infty)$, System~\eqref{eq_gantouh} is $L^q$-approximately controllable in time $\tau_{1}+\dotsb+\tau_n$ if and only if 
$\mathcal G$ is the disjoint union of directed cyclic graphs and 
\begin{equation}
\label{eq:app-control-graph}
    V(p)\cap \ker \Gamma^T= \{0\} \qquad \mbox{for every }p\in \mathcal{R}_{\mathcal{G}}.
\end{equation}
\end{theorem}

\begin{proof}
Thanks to Theorem~\ref{Th_appro_con} and Lemma~\ref{lem:rank-iff-cycles}, the theorem is proved if we show that Item~\ref{item:AC-1} from Theorem~\ref{Th_appro_con} is equivalent to \eqref{eq:app-control-graph}. By Lemma~\ref{lem:Dl}, if $p \notin \mathcal R_{\mathcal G}$, we have $\rank H_{\rm net}(p) = n$. The conclusion now follows since, for $p \in \mathcal R_{\mathcal G}$, we have $(\Ran [H_{\rm net}(p), B])^\perp = (\Ran H_{\rm net}(p))^\perp \cap (\Ran \Gamma)^\perp= (\Ran H_{\rm net}(p))^\perp \cap \ker \Gamma^T$ and $ \Ran H_{\rm net}(p) = V(p)^\perp$ by Lemma~\ref{lem:Dl}. 
\end{proof}

We deduce the following corollary in the case where $V(p)$ is of dimension one for every $p\in \mathcal{R}_{\mathcal{G}}$.

\begin{corollary}
Assume that $\mathcal G$ is the disjoint union of directed cyclic graphs and that, for $l_1\ne l_2$ in $\llbracket1,L\rrbracket$,
\[\frac{\sum_{j\in \mathcal{J}_{l_1}}\zeta_j}{\sum_{j\in \mathcal{J}_{l_1}}\tau_j}\ne \frac{\sum_{j\in \mathcal{J}_{l_2}}\zeta_j}{\sum_{j\in \mathcal{J}_{l_2}}\tau_j}.\]
Then, for $q \in [1,+\infty)$, System~\eqref{eq_gantouh} is $L^q$-approximately controllable in time $\tau_{1}+\dotsb+\tau_n$ if and only if 
$\tilde y_l(p)^T \Gamma\ne 0$ for every $l\in \llbracket1,L\rrbracket$ and $p\in \mathcal{D}_l$. 
\end{corollary}

We next turn to our exact controllability result for System~\eqref{eq_gantouh}. We will rely on Theorem~\ref{Th_exact_con1} and, for that purpose, we first extend Lemma~\ref{lem:Dl}.

\begin{lemma}\label{lem:Dl-closure}
Let $M\in \overline{H_{\rm net}(\mathbb C)}$ and let $(p_\nu)_{\nu\in \mathbb N}$ be such that $M=\lim_{\nu\to\infty}H_{\rm net}(p_\nu)$. 
In particular, for every $j \in \llbracket 1, n \rrbracket$, $e^{p_\nu\tau_j}$ converges as $\nu$ goes to infinity.
Then the following holds true.
\begin{enumerate}
    \item \label{item:uno-one}
    The matrix $M$ is singular if and only if 
    there exists $l\in \llbracket1,L\rrbracket$ such that $\lim_{\nu\to\infty}\Re(p_\nu)=-\frac{\sum_{j\in \mathcal{J}_{l_1}}\zeta_j}{\sum_{j\in \mathcal{J}_{l_1}}\tau_j}$
and $\Im(p_\nu)=k_\nu \frac{2\pi}{\sum_{j\in \mathcal{J}_{l_1}}\tau_j}+\eta_\nu$ with $k_\nu\in \mathbb Z$ and $\lim_{\nu\to \infty}\eta_\nu=0$.
   \item\label{item:due-two} 
   If $M$ is singular and $l\in \llbracket1,L\rrbracket$ is an in Item~\ref{item:uno-one}, then 
    rank of the $l$th block of $M$ is equal to $h_l-1$ and its range is equal to the orthogonal space 
    to  the vector $Y_l(M)=\lim_{\nu\to \infty}y_l(p_\nu)$, where, for $p\in \mathbb C$,   $y_l(p)$ is defined as in  \eqref{eq:zl}.
\end{enumerate}
\end{lemma}

The proof can be obtained by adapting that of Lemma~\ref{lem:Dl}. 

Let  $\tilde Y_l(M)\in \mathbb C^n$ denote the vector obtained from 
$Y_l(M)$ 
by identifying $\mathbb C^{h_l}$ with the subspace of $\mathbb C^n$ corresponding to the indices in $\mathcal J_l$. 
Reasoning as above, we can then obtain the following corollary of Theorem~\ref{Th_exact_con1}.

\begin{corollary}
Assume that $\mathcal G$ is the disjoint union of directed cyclic graphs and that, for $l_1\ne l_2$ in $\llbracket1,L\rrbracket$,
\[\frac{\sum_{j\in \mathcal{J}_{l_1}}\zeta_j}{\sum_{j\in \mathcal{J}_{l_1}}\tau_j}\ne \frac{\sum_{j\in \mathcal{J}_{l_2}}\zeta_j}{\sum_{j\in \mathcal{J}_{l_2}}\tau_j}.\]
Then System~\eqref{eq_gantouh} is $L^1$-exactly controllable 
in time $\tau_{1}+\dotsb+\tau_n$
if and only if 
$\tilde Y_l(M)^T \Gamma\ne 0$ for every $l\in \llbracket1,L\rrbracket$ and every $M\in \overline{H_{\rm net}(\mathbb C)}$ such that the $l$th diagonal block of $M$ is singular. 
\end{corollary}

\bibliographystyle{abbrv} 
\bibliography{ifacconf1}

\appendix

\section{Appendix}

We prove in this appendix the following statement, needed in the proof of Theorem~\ref{Th_exact_con1}.
Its proof follows that of Theorem~5.13 in \cite{ChitourHautus}.

\begin{proposition}
\label{prop_fond_exact_controllability_L1}
If System~\eqref{eq:hyperbolic} is $L^1$-exactly controllable in time 
$T_\ast = \tau_1 + \dotsb + \tau_n$
then there exist two matrices $R$ and $S$ with entries in $M(\mathbb{R}_{-})$ such that \eqref{eq:bezout} holds true.
\end{proposition}

\begin{proof}
Let $\mathcal{D}'_{+}(\mathbb{R})$ be the space of distributions having 
left-bounded
support, which becomes an algebra when endowed with the convolution product~$*$. Then the distribution $Q$ is invertible over $\mathcal D^\prime_+(\mathbb R)$ with respect to the convolution product, and its inverse is the matrix-valued Radon measure 
\[Q^{-1}=\diag(\delta_{\tau_1},\dots,\delta_{\tau_n})*   \sum_{k=0}^{+\infty} \left(   \sum_{j=1}^n K\diag(\delta_{\tau_1},\dots,\delta_{\tau_n}) \right)^{\ast k} 
,\]
where $M^{\ast k}$ denotes the convolution product of a matrix-valued Radon measure $M$ repeated $k$ times, with the convention $M^{* 0} = I_n \delta_0$. 

Let us define the map $\widetilde{G}\colon \widetilde{\Omega}^1 \longrightarrow \widetilde \Sigma^1$ by
\[(\widetilde{G}(u))_j (t)=(\pi(Q^{-1}*P*u))_j (t),\qquad j \in \llbracket 1, n\rrbracket,\; t\in [0,\tau_j],\; u \in \widetilde{\Omega}_1,\]
where $\widetilde \Sigma^1 = \prod_{j=1}^n L^1([0,\tau_j], \mathbb{R})$ and
  $\widetilde{\Omega}^1$ denotes the subspace of 
$\Omega^1$ made of inputs with 
compact support in $[-T_\ast,0]$ endowed with the norm $\|\cdot\|_{[-T_\ast,0],1}$. Firstly, we can see that the map $\widetilde{G}$ is a bounded linear operator because $Q^{-1}*P$ is a distribution with a finite number of Dirac distributions on each compact interval of $\mathbb{R}$. Secondly, note that System~\eqref{eq:hyperbolic} is $L^1$-exactly controllable in time $T_\ast$ if and only if the map $\widetilde{G}$ is surjective. We can now apply 
the open mapping theorem (see, e.g, \cite[Theorem~4.13]{rudin1991functional}) and deduce that there exists $\delta>0$ such that 
\begin{equation}
\label{eq:open_map_thereom}
\widetilde{G}(U) \supset \delta V,
\end{equation}
where $U$ and $V$ are the open unit balls of $\widetilde{\Omega}^1$ and $\widetilde \Sigma^1$ respectively.

Since 
$\pi(Q*\pi (Q^{-1}))=\pi(\delta_0)=0$ and the inclusion $X^{Q,2} \subset X^{Q,1}$ holds true, \cite[Lemma~4.3]{YamamotoRealization} implies that there exists a sequence 
$\psi_k=(\psi_{k,1},\dotsc,\psi_{k,n}) \in (X^{Q,1})^n$, $k \in \mathbb{N}$, such that $\psi_k \rightarrow \pi (Q^{-1})$ in the distributional sense as $k\to\infty$. Hence, for $i,j \in \llbracket 1, n\rrbracket$ and $k \in \mathbb N$, if we define $\left(Q^{-1}_{k}\right)_{i,j}$ as the Radon measure which is absolutely continuous with respect to the Lebesgue measure with density $\left(\psi_{k}\right)_{i,j}$,
we get that $\left(Q^{-1}_k\right)_{i,j}$ weak-star converges to $\left(\pi(Q^{-1})\right)_{i,j}$ in the sense of \cite[Paragraph~4.3]{maggi2012sets} as $k \to +\infty$. By \cite[Remark~4.35]{maggi2012sets}, we obtain that the total variation of $\left(Q^{-1}_k\right)_{i,j}$ over $[0, \tau_i]$, given by
\begin{equation*}
\left|\left(Q^{-1}_k\right)_{i,j}\right|([0,\tau_i]):= \int_0^{\tau_i} \left|\left(\psi_k\right)_{i,j}(t)\right| dt
\end{equation*}
is uniformly bounded with respect to $k \in \mathbb N$, i.e.,
\begin{equation*}
    \sup_{k \in \mathbb{N}}\left|\left(Q^{-1}_k\right)_{i,j}\right|([0,\tau_i])<\infty.
\end{equation*}
Therefore, 
there exists $C>0$ such that
\begin{equation*}
    \|\psi_{k,j}\|_{\widetilde \Sigma^1} \le C,\qquad j \in \llbracket 1,n\rrbracket,\ \ k \in \mathbb{N}.
\end{equation*}
Let $M'>0$ be such that $\delta M' >C$. We get from Equation~\eqref{eq:open_map_thereom} that
\begin{equation*}
\widetilde{G}(M'U) \supset \delta M'V
\end{equation*}
so that, for every $j \in \llbracket 1,n\rrbracket$ and $k \in \mathbb{N}$, there exists $S_{k,j} \in \widetilde{\Omega}_1$ such that
\begin{equation*}
    \widetilde{G}(S_{k,j})= \psi_{k,j} \quad \mbox{and} \quad  \|S_{k,j}\|_{[-T_\ast,0],\,1} \le M'.
\end{equation*}

Set $S_k=(S_{k,1},\dots,S_{k,n})$. 
 By construction, 
 $S_k \in {\widetilde{\Omega}}^1$ and 
\begin{equation}
\label{eq:bezout1}
\pi(Q^{-1}*P*S_k) \rightarrow \pi(Q^{-1}) \qquad \mbox{as $k\to\infty$},
\end{equation}
in distributional sense. Since the columns of $S_k$, for $k \in\mathbb{N}$, are uniformly bounded for the norm in $\widetilde{\Omega}^1$, by the weak compactness of Radon measures (see for instance \cite[Theorem 4.33]{maggi2012sets}), 
there exists a matrix $S$ with entries in $M(\mathbb{R}_-)$ 
such that, up to extracting a subsequence,  
$\lim_{k \rightarrow + \infty}S_k = S$ in distributional sense. Since the convolution is continuous in distributional sense and $\pi$ is continuous with respect to the strong dual topology, we deduce from Equation~\eqref{eq:bezout1} that
\begin{equation}
\label{eq:bezout2}
\pi(Q^{-1}*P*S)
=\pi (Q^{-1}).
\end{equation}

Let
 \begin{equation}
 \label{eq:bezout3}
  R:= Q^{-1}-Q^{-1}*P*S.
   \end{equation}
By Equation~\eqref{eq:bezout2}, we have $\pi(R)=0$, which, together with \cite[Lemma~A2]{YamamotoRealization}, implies that the support of $R$ is compact and contained in $(-\infty,0]$. Moreover, the entries of $R$ are distributions of order zero, since the convolution of distributions of order zero remain a distribution of order zero. Equation~\eqref{eq:bezout} is then obtained by left convoluting Equation~\eqref{eq:bezout3}  by $Q$.
\end{proof}

\end{document}